\documentclass[12pt]{amsart}
\usepackage[applemac]{inputenc}

\usepackage[T1]{fontenc}

\usepackage{ulem}
\usepackage{amsmath,esint,color}
\usepackage{amsthm,latexsym,epsfig,graphicx}
\usepackage{amsmath}
\usepackage{amsfonts}
\usepackage{amssymb}
\usepackage{bigints}
\usepackage{fullpage}
\usepackage[colorlinks=true, pdfstartview=FitV, linkcolor=blue, citecolor=blue, urlcolor=blue]{hyperref}

\date{}
\definecolor{sah}{rgb}{0.66,0.33, 0.04}
\definecolor{adel4}{cmyk}{1,0,0,0}
\definecolor{adel3}{rgb}{0.66,0.33, 0.04}
\definecolor{adel1}{cmyk}{0,0.20,1,0}
\definecolor{adel2}{cmyk}{0,0.40,1,0.30}
\definecolor{adel0}{rgb}{0.99,0.60, 0.30}
\definecolor{trut}{rgb}{0.99,0.80, 0.00}
\definecolor{trus}{rgb}{0.00, 0.50, 0.00}
 \definecolor{trust}{rgb}{0.99, 0.99, 0.80}
\definecolor{MaCouleur}{rgb}{0,0.9,0.3}

\newcommand{\CC}{\mathbb{C}}
\newcommand{\NN}{\mathbb{N}}
\newcommand{\RR}{\mathbb{R}}

\theoremstyle{plain}

\newtheorem{theorem}{Theorem}
\newtheorem{proposition}{Proposition}
\newtheorem{lemma}{Lemma}
\newtheorem*{teorem}{Main theorem}
\newtheorem{remark}{Remark}

\def\virgp{\raise 2pt\hbox{,}}

\def\Xint#1{\mathchoice
 {\XXint\displaystyle\textstyle{#1}}%
 {\XXint\textstyle\scriptstyle{#1}}%
 {\XXint\scriptstyle\scriptscriptstyle{#1}}%
 {\XXint\scriptscriptstyle\scriptscriptstyle{#1}}%
 \!\int}
\def\XXint#1#2#3{{\setbox0=\hbox{$#1{#2#3}{\int}$}
 \vcenter{\hbox{$#2#3$}}\kern-.5\wd0}}

\def\av_#1{\Xint-_{#1}}

\usepackage[textmath,displaymath,floats,graphics,delayed]{preview}
\title[]{Degenerate bifurcation of the rotating patches}

 \author[T. Hmidi]{Taoufik Hmidi}
\address{IRMAR, Universit\'e de Rennes 1\\ Campus de
Beaulieu\\ 35~042 Rennes cedex\\ France}
\email{thmidi@univ-rennes1.fr}
\author[J. Mateu]{Joan Mateu}
\address{Departament de Matem\`{a}tiques\\
Universitat Aut\`{o}noma de Barcelona\\
08193 Bellaterra, Barcelona, Catalonia} \email{ mateu@mat.uab.cat}

\begin{document}

\maketitle

%\date{}

%\subjclass[2000]{35Q35, 76B03, 76C05}
%\keywords{ 2D inviscid SQG, rotating patches, bifurcation theory}
%\begin{document}
\begin{abstract}{
In this paper we study  the existence of doubly-connected rotating patches  for Euler equations when the classical non-degeneracy conditions are not satisfied.  We prove   the bifurcation of the  V-states with two-fold symmetry, however for higher  $m-$fold symmetry with $m\geq3$ the bifurcation does not occur. This answers to a problem left open in \cite{H-F-M-V}. Note that, contrary to the known results  for  simply-connected and  doubly-connected cases where the bifurcation is pitchfork, we show that the  degenerate bifurcation is actually  transcritical.  These results are in agreement with the numerical observations recently discussed  in \cite{H-F-M-V}. The proofs stem from the local  structure of the quadratic form associated to the reduced bifurcation equation. } \end{abstract}
\tableofcontents

\section{Introduction}

 In this paper we deal with  the vortex motion for incompressible  Euler equations in two-dimensional space. The formulation velocity-vorticity is given by the nonlinear transport equation \begin{equation}\label{vorticity}
\left\{ \begin{array}{ll}
\partial_{t}\omega +v\cdot \nabla\, \omega=0,\,\, x\in \RR^2,\,\, t\geq0, \\
 v=\nabla^\perp\Delta^{-1}\omega,\\
\omega_{|t=0} =\omega_{0},
\end{array} \right.
\end{equation}
where  $v=(v^1,v^2)$ denotes  the  velocity field and $\omega=\partial_1 v^2-\partial_2 v^1$ its vorticity. The second equation in \eqref{vorticity} is the Biot-Savart law which can be written with a singular operator as follows:  By identifying  the vector $v=(v_1,v_2)$ with the complex function $v_1+iv_2$,  we may  write
\begin{equation}\label{Biot}
v(t,z)=\frac{i}{2\pi} \int_\CC
\frac{\omega(t,\xi)}{\overline{z}-\overline{\xi}}\,dA
(\xi),\quad z \in \CC,
\end{equation}
with $dA$ being the planar Lebesgue measure. Global existence of  classical solutions was established   a long time ago by Wolibner in \cite{Wo}  and  follows from the transport structure of the vorticity equation. For a recent account of the theory  we refer the reader to 
\cite{BM, Ch}. The same result was extended for less regular initial data by 
Yudovich in \cite{Y1} who proved  that the system \eqref{vorticity} admits   a unique global
solution in the weak sense when the initial vorticity $\omega_0$
is bounded an integrable.  This result is of great importance because it  enables  to deal rigorously with some discontinuous vortices as the  vortex patches which are the characteristic function of bounded domains. Therefore,  when $\omega_0={\chi}_{D}$ with $D$ a bounded domain then the solution of \eqref{vorticity} preserves this structure  for long  time  and $\omega(t)={\chi}_{D_t},$ where $D_t=\psi(t,D)$ being
the image  of $D$ by the flow. The motion of the patch is governed by the contour dynamics equation which takes the following form : Let $\gamma_t :\mathbb{T}\to \partial D_t$ be the Lagrangian  parametrization of the boundary, then
\begin{equation*}
\partial_t\gamma_t=-\frac{1}{2\pi}\int_{\partial D_t}\log|\gamma_t-\xi|d\xi.
\end{equation*}
Notice that according to Chemin's result \cite{Ch}, when the initial boundary is of H\"{o}lder class $C^{1+\varepsilon}$ with $0<\varepsilon<1,$ then there is no formation of singularities in finite time and  the  regularity is therefore  propagated for long time. In general, it is not an easy task  to tackle the full dynamics of a given vortex patch due the singular nonlinearity governing the motion of its boundary. Nevertheless, we know two explicit examples where the patches  perform a steady rotation about their centers without changing shape. The first one is   Rankine vortices where the boundary is a circle and this provides the only stationary patches according to Fraenkel's result  \cite{Fran}. As to the second example, it   is less obvious and goes back  to Kirchhoff  \cite{Kirc} who proved that  an elliptic patch rotates about its center with the angular velocity $\Omega = ab/(a+b)^2$, where $a$ and $b$ are the semi-axes of the ellipse. For another proof see   \cite[p.304]{BM}.
Many years later, 
several examples  of rotating patches, called also $V$-states, were constructed numerically  by  Deem and Zabusky \cite{DZ}  using contour dynamics algorithm. Note that a rotating patch is merely a patch whose domain $D_t$ rotates uniformly about its center which can be assumed to be the origin. This means that $D_t =e^{it \Omega} D$ and $\Omega\in \RR$ being  the angular velocity.  Later, Burbea provided in \cite{B}  an analytical proof and showed 
the  existence of the 
 V-states with  $m$-fold symmetry for each integer $m \geq 2$. In this countable family the case $m=2$ corresponds to Kirchhoff elliptic vortices.
% Recall that a domain is said  $m$-fold symmetric if it has the same group invariance of a regular polygon with $m$ sides. 
Burbea's approach consists in using complex analysis tools combined  with the bifurcation theory.  It should be noted that from this point of view  the rotating patches are arranged   in a  collection of countable curves bifurcating from Rankine vortices (trivial solution) at the discrete angular velocities set $\big\{\frac{m-1}{2m}, m\geq 2\big\}.$ The numerical experiments  of the   limiting V-states which are the ends of each branch were accomplished  in  \cite{Over,WOZ} and they  reveal interesting behavior; the boundary develops corners at right angles. Recently,  Verdera and the  authors studied   in \cite{HMV} the  regularity of the V-states close to the disc and proved    that the boundaries  are  in fact  of class   $C^\infty$ and convex. More recently, this result was improved by Castro, C\'ordoba and G\'omez-Serrano  in \cite{Cor22} who showed  the analyticity of the V-states close to the disc. We  mention   that similar research has been conducted  in the past few years  for more singular nonlinear transport equations arising in fluid dynamics  as the surface quasi-geostrophic equations or the   quasi-geostrophic shallow-water equations; see for instance \cite{Cor1,Cor22,Plotka,H-H,H-H-H}.

It should be noted that the bifurcating curves from the discs are pitchfork and always located in the left-hand side of the point of bifurcation $\frac{m-1}{2m}$. This implies that  angular velocities of those  V-states  are actually  contained in the interval $]0,\frac12[.$ Whether or not  other  V-states still survive  outside this range of  $\Omega$ is an interesting problem.  The limiting cases $\Omega\in\{0,\frac12\}$ are  well-understood and it turns out that the discs are   the only solutions for the V-states problem. As we have mentioned before, this was proved in the  stationary case $\Omega=0$  by Fraenkel  in  \cite{Fran}  using the  moving plane method. We  remark that this method is quite  flexible and was adapted by the first author  to $\Omega<0$ but  under a convexity restriction, see \cite{Hm}.
With regard to the second endpoint   $\Omega=\frac12$ it is somehow more easier than $\Omega=0$ and was solved recently  in \cite{Hm} using the maximum principle for harmonic functions.

The study of  the second  bifurcation of rotating patches from Kirchhoff ellipses was first examined by Kamm in \cite{Kam}, who gave numerical evidence of the existence of some branches bifurcating from the ellipses, see also \cite{Saf}.  In \cite{Luz},  Luzzatto-Fegiz and Willimason gave more details about the bifurcation diagram of  the first curves and where  we can also find some nice pictures  of the limiting V-states. Very recently, analytical proofs of the existence part of the  V-states and their regularity were explored in \cite{Cor22,H-M}. Indeed,  the authors showed  that  the bifurcation from the ellipses $\big\{w+Q\overline{w},\, w\in \mathbb{T}\big\}$ parametrized by the associated exterior conformal mapping occurs at the values $Q\in ]0,1[$ such that there exists  $m\geq3$ with
$$
1+Q^m-\frac{1-Q^2}{2} m=0.
$$ 
   The linear and nonlinear stability of  long-lived vortex structures  is an old subject in fluid dynamics and  much research  has been carried out  since Love's work \cite{Love} devoted to Kirchhoff ellipses. The nonlinear stability of the elliptic vortices was explored later in  \cite{Guo,Tang}. As to rotating patches   of m-fold symmetry,  a valuable discussion on the  linear stability   was done  by    Burbea and Landau in  \cite{Landau} using complex analysis approach. However the nonlinear stability  in a small neighborhood of Rankine vortices   was performed by Wan in \cite{Wan} through some  variational arguments. For further numerical  discussions, see \mbox{also \cite{Cerr,DR,Mit}.}
   
Recently, in  \cite{HMV2,H-F-M-V} close attention  was paid  to doubly-connected $V$-states taking the form  $\chi_{D}$ with   $D=D_1\backslash D_2$ and  $D_2\Subset D_1$ being two simply-connected domains. It is apparent from symmetry argument  that an annulus is a stationary  $V$-state, however and up to our knowledge  no other explicit example  is known in the literature. Note that in the paper 
\cite{HMV2} a full characterization of the V-states (with nonzero magnitude in the interior domain) with at least one elliptical interface was performed complementing the results of Flierl and Polvani \cite{Flierl}. As a by-product, the only elliptic doubly-connected V-states are  the annuli. The existence of nonradial doubly-connected V-states was  studied in \cite{H-F-M-V} in the spirit of Burbea's work but with much more involved calculations. Roughly speaking,  starting from  the Rankine patch supported by the annulus  
\begin{equation}\label{ann54}
\mathbb{A}_b\triangleq \{z\in \CC,  b < |z| < 1
\}
\end{equation}
  we  obtain  a collection of  countable bifurcating curves from this annulus at some explicit values of the angular velocities.  To be more  precise, let    $m\geq3$  verifying
\begin{equation}\label{strictsisc}
 1+{b}^{{m}}-\frac{1-{b}^2}{2} {m}<0,
\end{equation}
then  there exists two  curves of  doubly-connected rotating patches with  $m$-fold symmetry bifurcating from the annulus $\mathbb{A}_b$ at  the  angular velocities
\begin{equation}\label{fff34}
\Omega_m^{\pm} = \frac{1-b^2}{4}\pm
\frac{1}{2m}\sqrt{\Delta_m},\quad \Delta_m\triangleq\Big(\frac{m(1-b^2)}{2}-1\Big)^2
-b^{2m}.
\end{equation}
In  \cite{H-F-M-V}, it is  showed that  the condition \eqref{strictsisc} is equivalent to $\Delta_m>0$, in which case the eigenvalues are simple and  therefore the transversality assumption required by Crandall-Rabinowitz theorem (CR theorem for abbreviation) is satisfied. The main purpose of the current paper is to investigate the degenerate  case  corresponding to vanishing discriminant, that is, $\Delta_m=0.$ This problem was left open in \cite{H-F-M-V} because   the transversality assumption is no longer verified  and  CR theorem is therefore inefficient. As a matter of fact, the study of the linearized operator is not enough to understand the local structure of the nonlinear problem and answer to the bifurcation problem.  Before stating our result we need to introduce the following set $\mathbb{S}$ and describe its structure. This set will be carefully studied  in  \mbox{Section \ref{Sec2}}. Let
\begin{eqnarray}\label{SS}
\mathbb{S}&\triangleq&\Big\{(m,b)\in\NN^\star\times ]0,1[, \,\Delta_m=0\Big\}\\
\nonumber&=&\Big\{(2,b), b\in ]0,1[\Big\}\cup \Big\{(m,b_m),m\geq3\Big\},
\end{eqnarray}
where  for given $m\geq3,$ $b_m$ is  the unique solution in $]0,1[$ of the equation
$$
1+{b}^{{m}}-\frac{1-{b}^2}{2} {m}=0.
$$
It is not so hard to  show that the sequence $ (b_m)_{m\geq3}$ is strictly increasing and converges \mbox{to $1.$} Moreover one can easily check that
$$
b_3=\frac12\quad \hbox{and}\quad b_4=\sqrt{\sqrt{2}-1}\approx 0.6435.
$$
Notice that from the numerical experiments done in \cite{H-F-M-V}  we observe the formation of small loops from the branches emanating from  the eigenvalues $\Omega_m^+$ and $\Omega_m^-$ (with $m\geq3$) when these latter are close enough. It turns out that these loops become very small and shrink to a point which is the trivial solution when the distance between the  eigenvalues  goes to  zero. This suggests that in the degenerate case  the bifurcation does not occur.  We shall attempt at clarifying this numerical evidence from theoretical standpoint   by providing   analytical proof based upon the bifurcation theory which however still an efficient tool.  We shall also investigate  the two-fold symmetry where the situation seems to be  completely different from the preceding one and  rotating patches continue to exist.  We can now formulate our main results.
 \begin{teorem}\label{fundamental}
 The following assertions hold true.
 \begin{enumerate}
 \item
Let $b\in ]0,1[\backslash\big\{b_{2m}, m\geq2\big\}$, then there exists a curve of $2$-fold doubly-connected  V-states bifurcating from the annulus $\mathbb{A}_b$ at the angular velocity $\frac{1-b^2}{4}.$ Furthermore, the bifurcation is transcritical. 
\item Let $b=b_m$ for some $m\geq3.$ Then there is no bifurcation of $m-$fold V-states from  the annulus $\mathbb{A}_b$.
%\item If $b=b_m$ with $m$ an odd integer. Then there exists a curve of $2$-fold doubly-connected  V-states bifurcating from the annulus $\mathbb{A}_b$ at the angular velocity $\frac{1-b^2}{4}\cdot$
\end{enumerate}

\end{teorem}
Before giving some details about the proof some remarks are in order.
\begin{remark}
As it will be discussed in   the proof of the Main theorem  postponed in Section \ref{Sec55} we have not only one curve of bifurcation  in the two-fold case but actually two curves. However, from geometrical point of view they can be related to each other using a symmetry argument and therefore only one branch is significant.
\end{remark}
\begin{remark}
The existence of two-fold V-states for $b\in \big\{b_{2n}, n\geq 2\big\}$ remains unresolved. In this case the kernel is two-dimensional  and the transversality assumption is not satisfied. We believe that more work is needed  in order to explore the bifurcation in this very special  case.
\end{remark}

Now we shall sketch the general ideas of the proof of the Main theorem.  We start with a suitable  formulation of the V-states equations which was  set up  in the preceding works \cite{B, HMV2,H-F-M-V} through the use of the conformal mappings. If we denote by   $\phi_j$  the exterior conformal mapping of the domain $D_j$ then  from Riemann mapping theorem  we get,$$
\phi_j(w)=b_j w+\sum_{n\geq0}\frac{a_{j,n}}{w^n}, \quad \, a_{j,n}\in \RR,\quad \, b_1=1,\, b_2=b.
$$
The coefficients are assumed to be real since we shall look for domains which are symmetric with respect to the real axis.
Therefore the  domain $D$  which is assumed to be smooth, of class $C^1$ at least, rotates steadily with an angular velocity $\Omega$ if and only if 
\begin{equation*}
\forall \, w\in \mathbb{T},\,\,G_{j}(\lambda,\phi_1,\phi_2)(w) \triangleq \textnormal{Im}\bigg\{ \Big((1-\lambda)\overline{\phi_j(w)}+I(\phi_j(w))\Big)\, {w}\, {{\phi_j'(w)}}\bigg\}=0, \,\,j=1,2,
\end{equation*}
with 
$$
\lambda\triangleq1-2\Omega\quad \hbox{and}\quad  I(z)\triangleq\fint_{\mathbb{T}}\frac{\overline{z}-\overline{\phi_1(\xi)}}{z-\phi_1(\xi)}\phi_1^\prime(\xi)d\xi-\fint_{\mathbb{T}}\frac{\overline{z}-\overline{\phi_2(\xi)}}{z-\phi_2(\xi)}\phi_2^\prime(\xi)d\xi.
$$
Here we use the parameter $\lambda$ instead of $\Omega$ as in \cite{HMV2,H-F-M-V} because we found it   more convenient in the final formulae that we shall recall in the next few lines.  
We remind from  \cite{H-F-M-V} the following result : If we denote by $\mathcal{L}_{\lambda,b}$   the linearized operator around  the annulus defined by  $$
\mathcal{L}_{\lambda,b}h\triangleq\frac{d}{dt} G(\lambda,th)_{|t=0},\quad h=(h_1,h_2).
$$
Then it   acts as a matrix Fourier multiplier, namely, 
 for $${
h_j(w)=\sum_{n\geq1} \frac{a_{j,n}}{w^{nm-1}},}
$$ belonging to some function space  that will be precise  later
we have 
\begin{eqnarray*}%\label{ddff}
\mathcal{L}_{\lambda,b}h=\sum_{n\geq 1}M_{nm}(\lambda)\left( \begin{array}{c}
a_{1,n} \\
a_{2,n}
\end{array} \right)e_{nm},  \quad e_n=\hbox{Im} (\overline{w}^n).
\end{eqnarray*}
The matrix $M_n$ is explicitly  given for $n\geq1$  by the formula,
\begin{eqnarray*}%\label{Matr1}
 M_{n}(\lambda)&=&\begin{pmatrix}
n\lambda -1-n b^2 & b^{n+1} \\
 -b^n& b\big(n\lambda-n+1\big)
\end{pmatrix}.
\end{eqnarray*}
Throughout this paper we shall use the terminology of  "nonlinear eigenvalues" or simply eigenvalues to denote the values of $\lambda$ associated to the singular matrix $M_n$ for some $n$. 
Bearing in mind the definition of $\Delta_n>0$ given in \eqref{fff34} then it is straightforward that when $\Delta_n>0$  the eigenvalues are distinct and they coincide when $\Delta_n$ vanishes. In this latter  case $\lambda=\lambda_n=\frac{1+b^2}{2}$ and as it was stressed  in \cite{H-F-M-V} the transversality assumption is not satisfied and thereby  Crandall-Rabinowitz theorem is no longer useful. In this delicate context  the linearized operator is not sufficient to describe the local behavior of the nonlinear functional and therefore  one should move to higher expansions of the functional. The bifurcation without transversality is often referred in the literature as {\it the degenerate bifurcation.} There are some studies devoted to this case; in  \cite[Section I.16]{Kil} and \cite{Kiel} Kielh\"ofer obtained some results on the bifurcation but with some restrictive structure on the nonlinear functional. His approach consists in giving a complete  ansatz  for the  real eigenvalue $\mu(\lambda)$ which is a small perturbation of zero. The bifurcation is then formulated through the platitude degree of $\mu$.  In our context, it is not clear whether our singular functional satisfies the conditions listed in \cite{Kil}. The strategy that we shall follow is slightly bit different from the preceding  one; instead of looking for the expansion of the  eigenvalue $\mu(\lambda)$ we analyze the expansion at the second order of the reduced bifurcation equation and this will be enough for proving our main result. More precisely, using  Lyapunov-Schmidt reduction we transform the infinite-dimensional problem  into a two-dimensional one. This is done by using some projections intimately connected to the spectral structure  of the linearized operator and the V-states equations  reduces in fact to a finite-dimensional equation called {\it reduced bifurcation equation}. It takes the form,
$$
F_2(\lambda,t)=0, \quad F_2:\RR\times\RR\to\RR.
$$
In the degenerate case where $(m,b)$ belongs to the set $ \mathbb{S}$ defined in \eqref{SS}  the point $(\lambda_m,0)$ is in fact  a critical point for $F_2$ and therefore the resolvability  of   the reduced bifurcation equation should require an expansion of  $F_2$ at least at the second order around  the critical point. The computations are fairly complicated and long, and ultimately we find that  the quadratic form is non-degenerate. Accordingly, we   stop the expansion at this order and decide about the occurrence of the bifurcation using perturbative argument.  To be more precise,  we show the local structure,
$$
F_2(\lambda,t)=a_m (\lambda-\lambda_m)^2+b_m t^2+\big((\lambda-\lambda_m)^2+ t^2\big)\varepsilon(\lambda,t).
$$
with $a_m>0,\,\forall m\geq2$ and $\displaystyle{\lim_{(\lambda,t)\mapsto(\lambda_m, 0)}\varepsilon(\lambda,t)=0}$. Thus the existence of bifurcating curves is related  to the sign of the coefficient $b_m$ which reveals   behavior change  with respect to the frequency mode $m.$ Indeed, when $m\geq3$  we get always $b_m>0$ asserting that $F_2$ is strictly convex and thereby there is no bifurcation.  As to the case $m=2$ we obtain  $b_2<0$ implying that  the  bifurcation equation admits non-trivial solutions parametrized by two different curves emerging from the trivial solution.  By symmetry argument of the V-states equations we show that from geometrical standpoint the two curves describe in fact the same V-states meaning that we have only one interesting bifurcating curve.

To end this introduction we shall briefly discuss some numerical experiments concerning the  V-states and which are done for us by Francisco de la Hoz. In Figure \ref{flambda00} we plot the bifurcation diagram giving the dependence with respect to the angular velocity $\Omega$ of the first coefficients $a_{1,1}$ and $a_{2,1}$ of the conformal mappings $\phi_1$ and $\phi_2.$ 

\begin{figure}[!htb]
\center
\includegraphics[width=0.5\textwidth, clip=true]{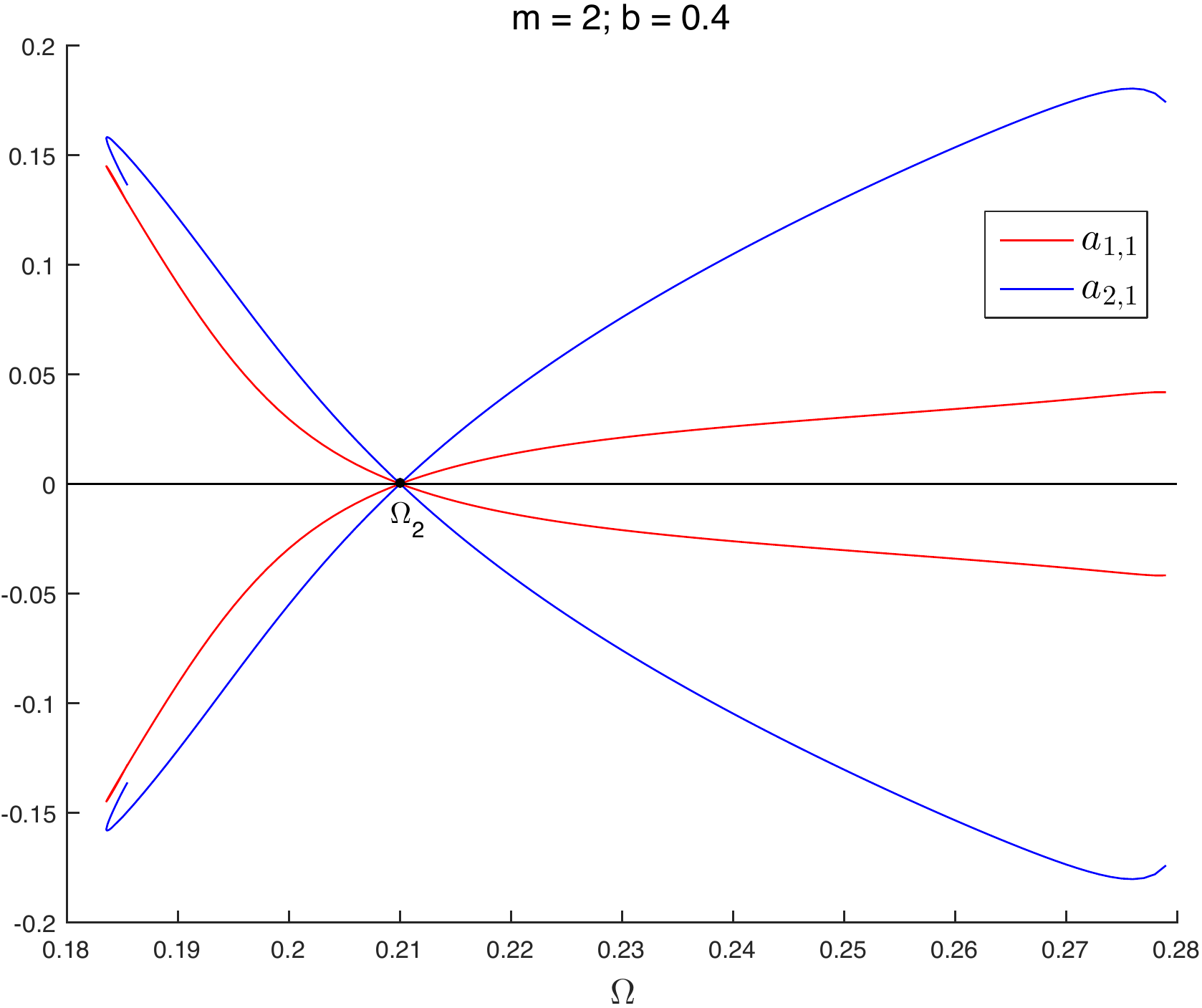}~
\caption{Bifurcation diagram.}
\label{flambda00}
\end{figure}

  In Figure \ref{flambda1} we plot the limiting V-states bifurcating from the annulus of small radius $b=0.4$. Note that there are two distinct  limiting V-states at the endpoints of the curve of bifurcation and  we observe in both cases the formation of two corner singularities at the inner curve, however the outer one remains smooth.
\begin{figure}[!htb]
\center
\includegraphics[width=0.5\textwidth, clip=true]{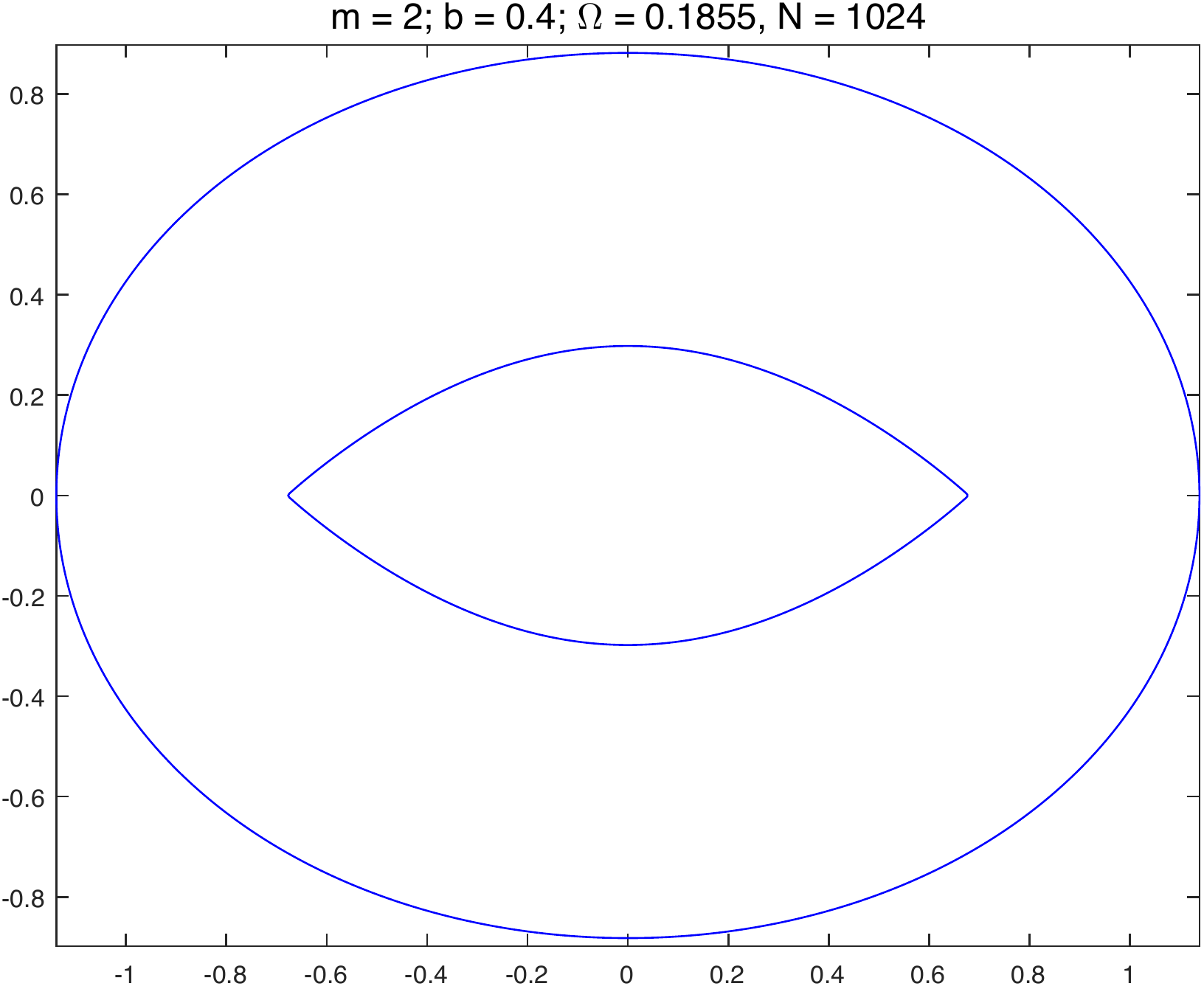}~
\includegraphics[width=0.5\textwidth, clip=true]{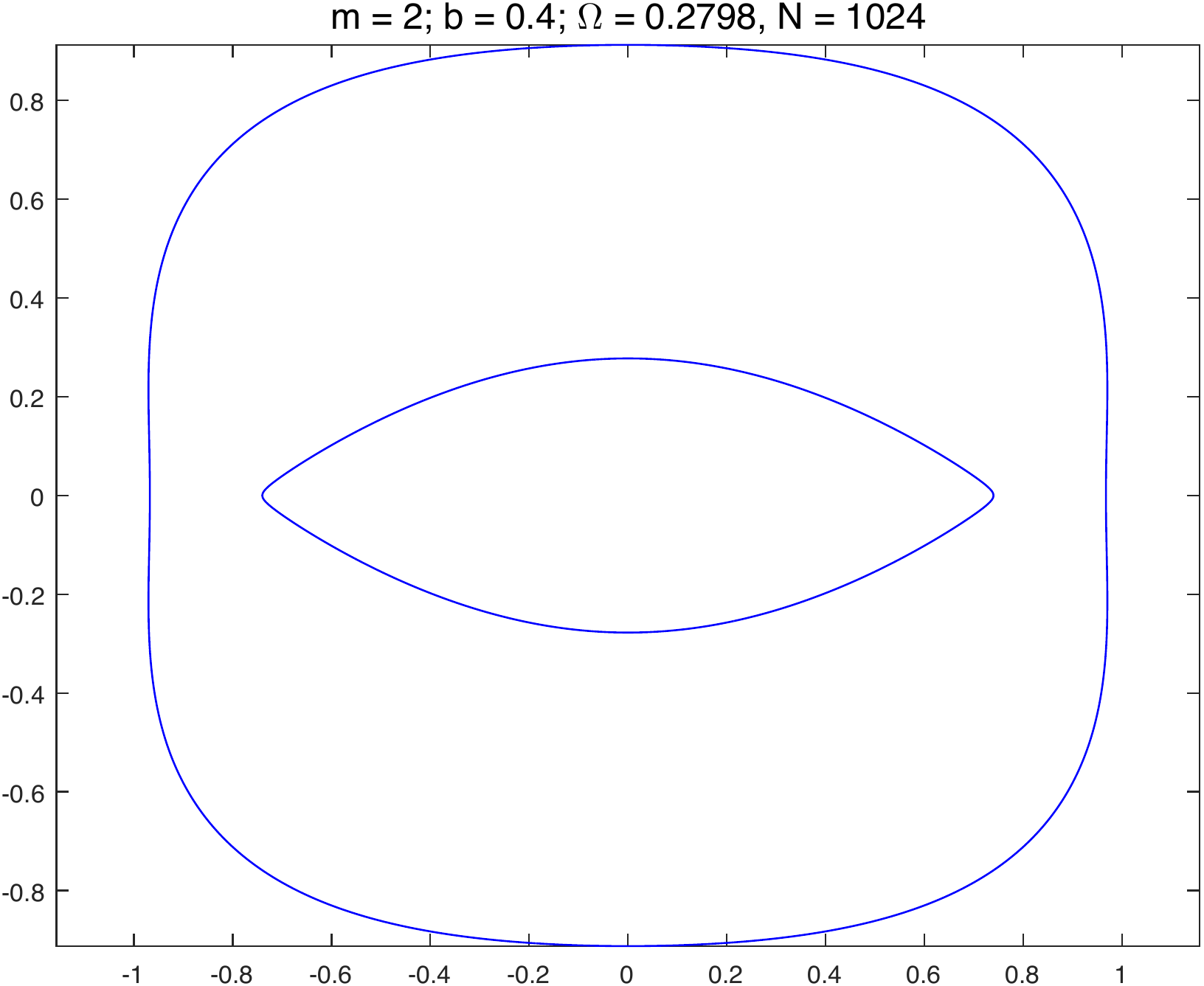}
\caption{Left:  the limiting V-state located  in the left side of branch of bifurcation. Right: the limiting V-state located  in the right side of the same branch.}
\label{flambda1}
\end{figure}

The paper will be organized as follows. In Section \ref{Sec33} we shall discuss the V-states equations and revisit the foundations of the stationary  bifurcation theory. In Section \ref{Sec-CR} we shall  deal with  the degenerate bifurcation and implement some general computations on the quadratic form associated to the reduced bifurcation equation. In section \ref{Sec44}, we implement explicitly  all the computations of the preceding section  in the special case of the V-states. The proof of \mbox{the Main theorem} will be given in Section \ref{Sec55}.
\vspace{0.3cm}

{\bf {Notation.}}
We need to collect some useful notation that will  be frequently used along this paper.
We shall use the symbol $\triangleq$ to define an object. Crandall-Rabinowitz theorem is sometimes shorten to CR theorem.
 %We denote by C any positive constant that may change from line to line.
The unit disc  is denoted by $\mathbb{D}$ and its boundary, the unit circle, by  $\mathbb{T}$. 
For given continuous complex function  $f:\mathbb{T}\to \CC$, we define  its  mean value by,
$$
\fint_{\mathbb{T}} f(\tau)d\tau\triangleq \frac{1}{2i\pi}\int_{\mathbb{T}}  f(\tau)d\tau,
$$
where $d\tau$ stands for the complex integration.

  Let $X$ and $Y$ be two normed spaces. We denote by $\mathcal{L}(X,Y)$ the space of  all continuous linear maps $T: X\to Y$ endowed with its usual strong topology. 
We shall denote by $Ker T$ and $R(T)$  the null space  and the range of $T$, respectively. Finally, if $F$ is a subspace of $Y$, then $Y/ F$ denotes the quotient space.

%%
%% \section{Preliminaries}\label{Sec33}
%%

 \section{Preliminaries and background}\label{Sec33}
 In this introductory section we shall gather some basic facts on the  V-states equations. Firstly, we shall write down the equations using   the conformal parametrization which appears to be much more convenient in the computation. Secondly,  we will focus on  the structure of the linearized operator around  the trivial solution and collect  some of its relevant   spectral properties. This is complemented by   Lyapunov-Schmidt reduction which allows us to perform the first step towards the bifurcation and leads to what is referred to as {\it reduced bifurcation equation}.
\subsection{Boundary equations}\label{bound548}
Let $D$ be a doubly-connected domain of the form $D=D_1\backslash D_2$ with $D_2\Subset D_1$ being  two simply-connected domains. Denote by  $\Gamma_j$ the boundary of the domain  $D_j$.  According to \cite{H-F-M-V},  the equation of the V-states rotating with the angular velocity $\Omega$ is  given by two coupled equations, one per each  boundary \mbox{component $\Gamma_j$.}
\begin{equation}\label{rotsqz1}
\,\textnormal{Re}\Big\{ \Big((1-\lambda)\overline{z}+ I(z)\Big)\, z^\prime\Big\}=0, \quad \forall\, z\in \Gamma_1\cup \Gamma_2,
\end{equation}
with $\lambda=1-2\Omega$ and 
$$
I(z)=\fint_{\Gamma_1}\frac{\overline{z}-\overline{\xi}}{z-\xi}d\xi-\fint_{\Gamma_2}\frac{\overline{z}-\overline{\xi}}{z-\xi}d\xi.
$$
The integrals are defined in the complex sense. Now we shall consider  the parametrization of the domains $D_j$  by the exterior conformal mappings: $\phi_j : \mathbb{D}^c \to D_j^c$ satisfying
\begin{equation*}
\phi_j(w)=b_jw+\sum_{n\geq 0}\frac{a_{j,n}}{w^n},\quad a_{j,n}\in \NN,
\end{equation*}
with $0<b_j<1$, $j=1,2$ and $b_2=b<b_1=1$. As we shall see later, this parametrization has many advantages especially in the computation of  the quadratic form associated  to the nonlinear   functional  through the use of residue theorem.  If the boundaries $\Gamma_j=\partial D_j$ are smooth enough then each conformal mapping admits univalent extension to the boundary $\mathbb{T}$, for instance see \cite{P,WS}. The restriction on $\Gamma_j$ is still denoted by $\phi_j$. After a change of variable in the integrals  we get
$$
I(z)=\fint_{\mathbb{T}}\frac{\overline{z}-\overline{\phi_1(\xi)}}{z-\phi_1(\xi)}\phi_1^\prime(\xi)d\xi-\fint_{\mathbb{T}}\frac{\overline{z}-\overline{\phi_2(\xi)}}{z-\phi_2(\xi)}\phi_2^\prime(\xi)d\xi.
$$
Setting $\phi_j=b_j \hbox{Id}+f_j$ then  the  equation \eqref{rotsqz1} becomes
$$
G(\lambda, f_1, f_2)=0
$$
with $G=(G_1,G_2)$ and 
\begin{equation}\label{G_j}
G_{j}(\lambda,f_1,f_2)(w) \triangleq \textnormal{Im}\bigg\{ \Big(2\Omega\overline{\phi_j(w)}+I(\phi_j(w))\Big)\, {w}\, {{\phi_j'(w)}}\bigg\}.
\end{equation}
We can easily check that
$$
G(\Omega,0,0)=0,\quad \forall \Omega\in \RR.
$$
This is coherent with the fact that an annulus is a stationary solution for Euler equations  and  due to the radial symmetry we can say that it rotates with any arbitrary angular velocity.
%%
%%\subsection{Linearized operator and Lyapunov-Schmidt reduction}
%%

\subsection{Linearized operator and Lyapunov-Schmidt reduction}\label{Sec2}
In this section we shall recall some results from \cite{H-F-M-V} concerning the structure of the linearized operator around the trivial solution of the functional $G$ defining the V-states equation \eqref{G_j}. The case of double eigenvalues was left open in  \cite{H-F-M-V} because we lose the transversality assumption in Crandall-Rabinowitz theorem. Our main focus is to make some  preparations in order to understand this degenerate case. The strategy is to come back to the foundation of the bifurcation theory and try to improve these tools. As we shall see,   Lyapunov-Schmidt reduction is used to  transform the infinite-dimensional problem into a two-dimensional one. Roughly speaking, after some  suitable projections we reduce  the  V-states equation  in local coordinates into an equation of the type
$$
F_2(\lambda,t)=0,\quad  F_2:\mathbb{R}^2\to\mathbb{R}.
$$
The points where  the transversality is not satisfied correspond merely  to critical points for $F_2.$ It turns out that the quadratic form associated to $F_2$ giving Taylor expansion at the second order around the critical points is non-degenerate and its signature depends on the frequency $m$. This will be enough   to answer to the  bifurcation problem. 
% The  formula for the Hessian of $F_2$ which are very long and highly computational will be discussed in \mbox{Section \ref{Sec-CR}.} 
Now we shall  introduce the function spaces that will capture the $m-$fold structure of the V-states.   
\begin{equation}\label{X_m}
X_m=\Big\{f=(f_1,f_2)\in (C^{1+\alpha}(\mathbb{T}))^2,\, f(w)=\sum_{n=1}^{\infty}A_{n}\overline{w}^{nm-1},\, A_n\in \mathbb{R}^2  \Big\}.
\end{equation}
\begin{equation}\label{Y_m}
Y_m=\Bigg\{G=(G_1,G_2)\in (C^{\alpha}(\mathbb{T}))^2,\, G=\sum_{n\geq 1}B_{n}e_{nm},\,\, 
B_{n}\in\mathbb R^2\Bigg\}, e_n(w)=\hbox{Im}(\overline{w}^n).
\end{equation}
The proof of the Main theorem consists in solving  the V-states equation in a small neighborhood of zero, namely, 
\begin{equation}\label{V-stat11}
G(\lambda,f)=0,\quad 
f=(f_1,f_2)\in B_{r}^m\times B_r^m\subset X_m,
\end{equation}
where we  denote by $B_r^m$  the ball of radius $r\in (0,1)$ in the space of functions which are one component of an element in $X_m$. 
The linearized operator around zero is  defined by 
$$
\mathcal{L}_{\lambda,b}h\triangleq D_fG(\lambda,0)h= \frac{d}{dt}G(\lambda,th)_{|t=0}
$$
and from \cite{H-F-M-V}, we assert that for $h=(h_1, h_2)\in X$  with the expansions
$$
h_j(w)=\sum_{n\geq1} \frac{a_{j,n}}{w^{nm-1}},
$$
one has
\begin{eqnarray}\label{ddff}
\mathcal{L}_{\lambda,b}h=\sum_{n\geq 1}M_{nm}(\lambda)\left( \begin{array}{c}
a_{1,n} \\
a_{2,n}
\end{array} \right)e_{nm}.
\end{eqnarray}
where the matrix $M_n$ is given for $n\geq1$  by
\begin{eqnarray}\label{Matr1}
 M_{n}(\lambda)&=&\begin{pmatrix}
n\lambda -1-n b^2 & b^{n+1} \\
 -b^n& b\big(n\lambda-n+1\big)
\end{pmatrix}.
%&\triangleq&\begin{pmatrix}
%a_n & b \,b_n \\
% -b_n& b \, c_n
%\end{pmatrix}.
\end{eqnarray}
The nonlinear eigenvalues  are the values of $\lambda$ such that at least one matrix $M_n$  is not invertible. Note the determinant equation is a polynomial of second order and it admits real roots if and only if the discriminant is positive. According to \cite[p.28]{H-F-M-V}, the roots are explicitly given by 
$$
\lambda_n^\pm=\frac{1+b^2}{2}\pm\frac1n\sqrt\Delta_n,
$$ 
with
\begin{equation}\label{Delta2}
\Delta_n= b^{2n}-\left( \frac{1-b^2}{2} n-
1\right)^2\geq0.
\end{equation}
 Bifurcation study when the roots are distinct, that is $\Delta_n>0$, was studied in \cite{H-F-M-V}. All the assumptions of Crandall-Rabinowitz theorem are satisfied and we have two distinct branches of solutions with $n-$fold symmetry. However, when $\Delta_n=0$ for some $n\geq2$  the transversality assumption is not satisfied  and  therefore CR theorem ceases to give  a profitable  result.  In Figure \ref{fbvs}, we draw for each $n\geq2$ the curve $\mathcal{C}_n$ which is the union of the curves  $b\in [0,b_n^\star]\mapsto \lambda_n^\pm$. Note that for  $n\geq3$, the curve associated to the sign  $+$ is  lying above  and communicates with the curve  associated to the sign $-$ at exactly  the turning  point whose abscissa is $b_n^\star$. This latter number giving the lifespan of the curves  coincides with $b_n$ defined below  in \eqref{relat1}. However for $n=2$,  $\Delta_2$ vanishes  for all the values of $b$ and therefore $\lambda_2^+=\lambda_2^-=\frac{1+b^2}{2}.$ Consequently  this eigenvalue is strictly increasing with respect to $b$ and it   cuts the curve $\mathcal{C}_n$ for each $n\geq3$ at the turning point (with abscissa $b_n$). This illustrates the special dynamics of $\lambda_2$ compared to higher ones and provides  somehow  insight into why the bifurcation occurs only for two-fold. 
  \begin{figure}[!htb]
\center
\includegraphics[width=0.5\textwidth, clip=true]{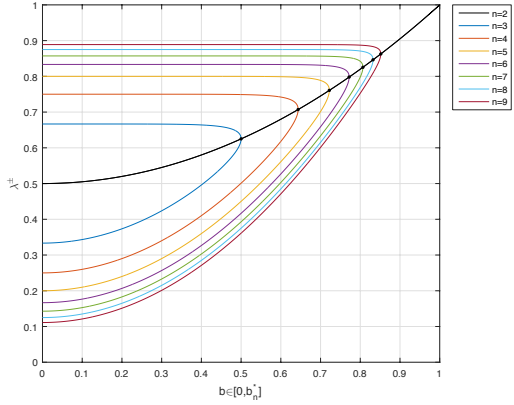}
\caption{\small{The curves of $\mathcal{C}_n$ giving  $ \lambda_n^\pm$ in function of $b$.}}
\label{fbvs}
\end{figure}
 As we have  already mentioned in the Introduction, the main target  of this paper is to focus on this
  degenerate case  where $\Delta_m=0$ for some $m\in \NN^\star$ and thus    $\lambda_n^+=\lambda_n^-=\frac{1+b^2}{2}$. The first step to settle this case is to understand the structure  of  the following set 
  $$\mathbb{S}\triangleq\Big\{(m,b)\in\NN^\star\times ]0,1[, \,\, \Delta_m=0\Big\}.
  $$
  It is plain that no solution is associated to $m=1$. However for $m=2,$ any $b\in]0,1[$ gives rise to an element of $\mathbb{S}$. As to the values $m\geq3$ we can easily check that the map $b\mapsto 1+{b}^{{m}}-\frac{1-{b}^2}{2} {m} $ is strictly increasing and therefore it admits only one solution denoted by  $b_m$, that is,
  \begin{equation}\label{relat1}
\frac{1-b_m^2}{2} m-1=b_m^m.
\end{equation}
It is not difficult to check that  this sequence is strictly increasing and converges to $1.$ Hence we deduce the decomposition,
 
$$\mathbb{S}=\Big\{(2,b), b\in ]0,1[\Big\}\cup \Big\{(m,b_m),m\geq3\Big\}.
$$
%$$
%\hbox{det }M_m(\lambda_m)=b\Big\{-\Big(\frac{1-b^2}{2}m-1\Big)^2+b^{2m}\Big\}=0.
%$$
For a future use we need to make  the notation,
\begin{equation}\label{lambda6}
\lambda_m=\left\{ \begin{array}{ll}
\frac{1+b^2}{2},&\hbox{if} \quad m=2, \, b\in ]0,1[\\
\frac{1+b_m^2}{2},& \hbox{if} \quad m\geq3.
\end{array} \right. 
\end{equation}
Now we wish to  conduct a spectral study for the matrix $M_m(\lambda_m)$ that will be used later  to deduce some spectral properties of the linearized operator $\mathcal{L}_{\lambda,b}.$ For $m=2$ direct computations yield
$$
M_2(\lambda_2)=b^2\begin{pmatrix}
-1 & b  \\
 -1& b 
\end{pmatrix}.
$$
However, for $m\geq3$   we can easily check that 
$$
M_m(\lambda_m)=b^m\begin{pmatrix}
1 & b  \\
 -1& -b 
\end{pmatrix}.
$$
We may  unify the structure of these matrices as follows, \begin{equation}\label{formm}
M_m(\lambda_m)=b^m\begin{pmatrix}
-\varepsilon & b  \\
 -1& \varepsilon \, b
\end{pmatrix},\quad \varepsilon=\left\{ \begin{array}{ll}
1,&\hbox{if} \quad m=2\\
-1,& \hbox{if} \quad m\geq3.
\end{array} \right. 
\end{equation}
To lighten the notation we prefer dropping the subscript   on  $b_m$ which is noted  simply  by $b$. We can show without any difficulty that the   kernel of $M_m(\lambda_m)$  is one-dimensional and generated by the vector 
$
\left( \begin{array}{c}
\varepsilon b\\
1
\end{array} \right).$
Now take $m\geq2, b=b_m$ and $\lambda=\lambda_m$ then the kernel of the operator $\mathcal{L}_{\lambda_m, b}$ defined by \eqref{ddff} is one-dimensional in $X_m$ and generated by
\begin{equation}\label{kernexp}
w\in \mathbb{T}\mapsto v_m(w)=\left( \begin{array}{c}
  \varepsilon b\\
1
\end{array} \right)\overline{w}^{m-1}.
\end{equation}
However for $m=2$ and $b\notin \big\{b_m, m\geq3\big\}$ the kernel of $\mathcal{L}_{\lambda_2, b}$ is also one-dimensional in $X_2$ and generated by the same vector \eqref{kernexp} with $\varepsilon=1$. 
As to the case $m=2$ and $b=b_n$ for some $n\geq3$, then the kernel is one-dimensional in $X_2$ if $n$ is odd (since the vector $v_n\notin X_2$) but two-dimensional if $n$ is even. To formulate in a compact way the one-dimensional case we shall introduce the set
\begin{equation}\label{dim11}
\widehat{\mathbb{S}}\triangleq\mathbb{S}\backslash \big\{(2,b_{2m}), m\geq2\big\}.
\end{equation}
Therefore if $(m,b)\in \widehat{\mathbb{S}}$ then the kernel  of $\mathcal{L}_{\lambda_m, b}$ is one-dimensional in $X_m$ and generated \mbox{by  \eqref{kernexp}.}
In conclusion,   the foregoing results can be summarized in the next proposition as follows.
\begin{proposition}\label{strutz23}
Let $(m,b)\in \mathbb{S}$ then the following assertions hold true.
\begin{enumerate}
\item If $(m,b)\in \widehat{\mathbb{S}}$ then the kernel of $\mathcal{L}_{\lambda_m, b}$ is one-dimensional in $X_m$ and generated by the vector defined in  \eqref{kernexp}.
\item If $m=2$ and $b=b_{2n}$ for some $n\geq2$ then  the kernel of $\mathcal{L}_{\lambda_2, b}$ is two-dimensional \mbox{in $X_2$.}

\end{enumerate}
\end{proposition}

Our next task   is to implement   Lyapunov-Schmidt reduction used as a basic tool  in  the bifurcation theory. One could revisit this procedure in a general setup but for the sake of simplicity we prefer limit the description to our context of the V-states. Let $(m,b)\in \widehat{\mathbb{S}}$ and 
denote by $\mathcal{X}_m$ any  complement of the kernel $\langle v_m\rangle$ in $X_m$. From now on we shall work with the following candidate : 
\begin{equation}\label{fs1} 
h\in \mathcal{X}_m\Longleftrightarrow 
h\in (C^{1+\alpha}(\mathbb{T}))^2,\,
 h(w)=\sum_{n\geq2}{A_n}\overline{w}^{nm-1}+\alpha\left( \begin{array}{c}
  1 \\
0
\end{array} \right)\overline{w}^{m-1}, A_n\in \mathbb R^2, \alpha\in\RR.
\end{equation}
It is not hard to check that this sub-space is complete and
$$
X_m=\langle v_m\rangle \oplus\mathcal{X}_m.
$$
Now denote by $\mathcal{Y}_m$ the range of $D_fG(\lambda_m,0)$ in $Y_m$ which is explicitly described by : 
%a function $k\in (C^{\alpha}(\mathbb{T}))^2$ belongs to $\mathcal{Y}_m$ iff
\begin{equation}\label{fs10}
k\in \mathcal{Y}_m \Longleftrightarrow k\in (C^{\alpha}(\mathbb{T}))^2,\,\,\, k(w)=\sum_{n\geq2}{B_n}e_{nm}+\beta\left( \begin{array}{c}
  \varepsilon \\
1
\end{array} \right)e_m, \quad B_n\in \RR^2, \,\beta\in\mathbb{R}.
\end{equation}
Then we observe that it is of co-dimension one and admits as a complement the line   
 $\langle \mathbb{W}_m\rangle $  generated by the vector %defined by  be a complement of $\mathcal{Y}_m$ in $Y_m.$  For example we make the choice 
\begin{eqnarray*}%\label{kernexpp}
 \mathbb{W}_m&=&\left( \begin{array}{c} 
  \frac{1}{\sqrt{2}} \\
-\frac{\varepsilon}{\sqrt{2}}
\end{array} \right) \, e_m\\
&\triangleq&\widehat{\mathbb{W}}\, e_m.
\end{eqnarray*}
Thus we may write 
$$
 Y_m=\langle  \mathbb{W}_m\rangle \oplus\mathcal{Y}_m.
$$
Consider  two projections $P: X_m\mapsto\langle v_m\rangle$ and $Q:  Y_m\mapsto \langle  \mathbb{W}_m\rangle$. 
As we shall see the structure of $P$ does not matter in the computations and therefore there is no need to work with an explicit element. However for the projection $Q$ we need to get an explicit one. This will be given  by the  "orthogonal projection", that is,  for $h\in Y_m$ defined in \eqref{Y_m} we set
\begin{equation}\label{project1}
h(w)=\sum_{n\geq1} B_n e_{nm},\quad Qh(w)=\big\langle B_1,\widehat{\mathbb{W}}\big\rangle  \mathbb{W}_m
\end{equation}
and $\langle,\rangle$ denotes the Euclidean scalar product of $\mathbb R^2.$ By virtue of the definition we observe that
\begin{equation}\label{annul}
Q\partial_fG(\lambda_m,0)=0.
\end{equation}
For $f\in X_m,$ we define  the components,
 $$g=Pf\quad\hbox{and}\quad  k=(\hbox{Id}-P)f.
 $$
   Then the V-states equation \eqref{V-stat11} is equivalent to the system
$$
F_1(\lambda,g, k)\triangleq(\hbox{Id}-Q)G(\lambda,g+k)=0\quad \hbox{and}\quad QG(\lambda, g+k)=0.
$$
The function $F_1:\RR\times  \langle v_m\rangle\times \mathcal{X}_m\to\mathcal{Y}_m$  is well-defined and smooth. The regularity follows from the study done in \cite{H-F-M-V}. Moreover, it is not difficult to check that 
\begin{equation}\label{invert1}
D_k F_1(\lambda_m,0,0)=(\hbox{Id}-Q)\partial_fG(\lambda_m,0).
\end{equation}
Using  \eqref{annul} we deduce that
$$D_k F_1(\lambda_m,0,0)=\partial_fG(\lambda_m,0): \mathcal{X}_m\to \mathcal{Y}_m,
$$  
which is invertible and by  \eqref{formm} we can explicitly get its inverse :
\begin{equation}\label{Inv}
\partial_fG(\lambda_m,0)h=k\Longleftrightarrow \forall n\geq2, A_n=M_{nm}^{-1} B_n\quad  \hbox{and}\quad \alpha=-\frac{\beta}{b^m}\cdot
\end{equation}
By the Implicit Function Theorem the solutions of the equation $F_1(\lambda_m,g,k)=0$ are described near the point $(\lambda_m,0)$ by the parametrization $k=\varphi(\lambda, g)$ with 
$$\varphi:\mathbb R\times \langle v_m\rangle\to\mathcal{X}_m.
$$ Therefore the resolution of V-states equation close to $(\lambda_m,0)$ is equivalent to  
$$
QG\big(\lambda, t v_m+\varphi(\lambda,t v_m)\big)=0.
$$
From the assumption of CR theorem $G(\lambda,0)=0, \,\forall \lambda$, one can deduce that
\begin{equation}\label{zero1}
\varphi(\lambda,0)=0,\quad\forall \lambda \quad \hbox{close to}\quad \lambda_m.
\end{equation}
Using  Taylor formula on the variable $t$  in order to get rid of the trivial solution corresponding to $t=0$, then  nontrivial solutions to \eqref{V-stat11} are locally described close to  $(\lambda_m,0)$ by the equation
\begin{equation}\label{Bifur1}
F_2(\lambda,t)\triangleq\int_0^1Q\partial_fG\big(\lambda, st v_m+\varphi(\lambda,st v_m)\big)\big(v_m+\partial_g\varphi(\lambda,st v_m) v_m\big)ds=0.
\end{equation}
This equation is called along this paper bifurcation equation. Moreover, using the relation \eqref{annul} we find
$$
F_2(\lambda_m,0)=0.
$$

%
%According to Taylor formula we get
%\begin{equation}\label{Bifur1}
%F_2(\lambda,t)\triangleq\int_0^1Q\partial_fG\big(\lambda, st v_m+\varphi(\lambda,st v_m)\big)\big(v_m+\partial_g\varphi(\lambda,st v_m) v_m\big)ds=0
%\end{equation}
%Observe that this a finite -dimensional equation since \quad $F_2:\mathbb R \times  \mathbb R \to  \langle  \mathbb{W}_m\rangle$. Moreover
%$$
%F_2(\lambda_m,0)=0.
%$$
%
%
\section{Degenerate bifurcation} \label{Sec-CR}

In this section we intend to compute the Jacobian and the Hessian of $F_2$ in a general framework. This  allows to go beyond  Crandall-Rabinowitz theorem and give profitable information when   the  transversality condition is not satisfied. To be precise about the terminology, we talk about {\it degenerate bifurcation} when the bifurcation occurs without transversality assumption. %In  the current context of the V-states  the quadratic form is not degenerate and it is enough to describe the local structure of the solutions to the reduced  bifurcation equation. 
First  we  recall Crandall-Rabinowitz theorem \cite{CR} which  throughout this paper is often denoted by CR theorem.
\begin{theorem}\label{C-R} Let $X, Y$ be two Banach spaces, $V$ a neighborhood of $0$ in $X$ and let 
$
F : \RR \times V \to Y
$
with the following  properties:
\begin{enumerate}
\item $F (\lambda, 0) = 0$ for any $\lambda\in \RR$.
\item The partial derivatives $F_\lambda$, $F_x$ and $F_{\lambda x}$ exist and are continuous.
\item The spaces $\hbox{Ker }\mathcal{L}_0$ and $Y/R(\mathcal{L}_0)$ are one-dimensional. 
\item {\it Transversality assumption}: $\partial_\lambda\partial_xF(0, 0)x_0 \not\in R(\mathcal{L}_0)$, where
$$
\hbox{Ker }\mathcal{L}_0 = span\{x_0\}, \quad \mathcal{L}_0\triangleq \partial_x F(0,0).
$$
\end{enumerate}
If $Z$ is any complement of $\hbox{Ker }\mathcal{L}_0$ in $X$, then there is a neighborhood $U$ of $(0,0)$ in $\RR \times X$, an interval $(-a,a)$, and continuous functions $\varphi: (-a,a) \to \RR$, $\psi: (-a,a) \to Z$ such that $\varphi(0) = 0$, $\psi(0) = 0$ and
$$
F^{-1}(0)\cap U=\Big\{\big(\varphi(\xi), \xi x_0+\xi\psi(\xi)\big)\,;\,\vert \xi\vert<a\Big\}\cup\Big\{(\lambda,0)\,;\, (\lambda,0)\in U\Big\}.
$$
\end{theorem}
We adopt in this section the same notation and definitions of the preceding section. Notice that CR theorem is an immediate consequence of the implicit function theorem applied to the function $F_2$ introduced in \eqref{Bifur1}. Indeed,  by virtue of  Proposition \ref{platit2}-(1) we have
$$
\partial_\lambda F_2(\lambda_m,0)=Q\partial_\lambda\partial_f G(\lambda_m,0)v_m.
$$ 
It is apparent that the  transversality assumption  is equivalent to saying  $\partial_\lambda F_2(\lambda_m,0)\neq0$ and thus we can apply the implicit function theorem. Unfortunately, this condition is no longer satisfied under the conditions of the Main theorem and we should henceforth work with the following assumption :\begin{equation}\label{lostra}
Q\partial_\lambda\partial_fG(\lambda_m,0)v_m=0.
\end{equation}
This means that $\partial_\lambda F_2(\lambda_m,0)=0$ and we shall see later in Proposition \ref{explicitval} that  $\partial_t F_2(\lambda_m,0)$ is also vanishing. The existence of nontrivial solutions to $F_2$ will be answered through the study of its Taylor expansion at the second order. 
In the next lemma we shall give some relations on the first derivatives of the function $\varphi.$

\begin{lemma}\label{platit1}
Assume that the function $G$  satisfies the conditions $(i)-(ii)-(iii)$ of CR theorem. Then 
$$
\partial_\lambda\varphi (\lambda_m,0)=\partial_g\varphi(\lambda_m,0)=0.
$$
In addition  if  $G$ is $C^2$ and satisfies $\partial_{\lambda\lambda}G(\lambda_m,0)=0$ then 
$$
\partial_{\lambda\lambda}\varphi (\lambda_m,0)=0.
$$

\end{lemma}
\begin{proof}
The first result follows easily from \eqref{zero1}. Concerning the second one, we rewrite  the equation of $F_1$ previously introduced  in Section \ref{Sec2}  in a neighborhood of $(\lambda_m,0)$ as follows,
\begin{equation}\label{th11}
(\hbox{Id}-Q)G\big(\lambda, t v_m+\varphi(\lambda,t v_m)\big)=0.
\end{equation}
Differentiating this equation  with respect to $t$  we get
$$
(\hbox{Id}-Q)\partial_fG\big(\lambda_m,0\big)\big(v_m+\partial_g\varphi(\lambda_m,0) v_m\big)=0.
$$
From \eqref{annul}  and  since $v_m\in \hbox{Ker } \partial_fG\big(\lambda_m,0\big)$ we may obtain
$$
\partial_fG\big(\lambda_m,0\big)\big(\partial_g\varphi(\lambda_m,0) v_m\big)=0.
$$
As by definition $(\partial_g\varphi(\lambda_m,0) v_m\in \mathcal{X}_m$ which a complement of the kernel of $\partial_fG\big(\lambda_m,0\big)$ then we get
$$
\partial_g\varphi(\lambda_m,0) v_m=0.
$$
It remains to check the last identity. Differentiating the equation \eqref{th11} twice  with respect to $\lambda$ we get
\begin{eqnarray*}
0&=&(\hbox{Id}-Q)\partial_{ff}G\big(\lambda_m,0\big)\Big[\partial_\lambda\varphi(\lambda_m,0),\partial_\lambda\varphi(\lambda_m,0)\Big]\\
&+&2(\hbox{Id}-Q)\partial_f\partial_\lambda G\big(\lambda_m,0\big)\partial_\lambda\varphi(\lambda_m,0)+(\hbox{Id}-Q)\partial_fG\big(\lambda_m,0\big)\partial_{\lambda\lambda}\varphi(\lambda_m,0)\\
&+&(\hbox{Id}-Q)\partial_{\lambda\lambda}G\big(\lambda_m,0\big).
\end{eqnarray*}
In the preceding identity we have used the following notation: for $v,w\in X$, $\partial_{ff}G[v,w]$ denotes  the value of the second derivative at the vector $(v,w).$ Using the relations $\partial_\lambda\varphi(\lambda_m,0)=\partial_g\varphi(\lambda_m,0)=0$ combined with the fact $\partial_{\lambda\lambda}G(\lambda_m,0)=0$ we get
$$
(\hbox{Id}-Q)\partial_fG\big(\lambda_m,0\big)\partial_{\lambda\lambda}\varphi(\lambda_m,0)=0.
$$
By \eqref{annul} we conclude that $\partial_{\lambda\lambda}\varphi(\lambda_m,0)\in \hbox{Ker }\partial_fG\big(\lambda_m,0\big)$ and we know from the  image of $\varphi$ that this vector belongs also to the vector space $\mathcal{X}_m$ which is a complement of the kernel. Therefore we deduce that
$$
\partial_{\lambda\lambda}\varphi(\lambda_m,0)=0.
$$
This achieves the proof.
\end{proof}
Next we shall be concerned with  some basic formulae for the  Jacobian and the Hessian  of the function $F_2$ introduced in \eqref{Bifur1}.
 \begin{proposition}\label{platit2}
 Assume that $G$ is sufficiently smooth ( of class $C^3$) and satisfies the conditions $(i)-(ii)-(iii)$ of CR theorem. Then the following assertions hold true.
\begin{enumerate}
\item First derivatives: 
$$
\partial_{\lambda}F_2(\lambda_m,0)=Q\partial_\lambda \partial_fG\big(\lambda_m,0\big)v_m.
$$
\begin{eqnarray*}
2\partial_t F_2(\lambda_m,0)&=& Q\partial_{ff}G(\lambda_m,0) \big[v_m,v_m\big]\\
&=&\frac{d^2}{dt^2} QG(\lambda_m, tv_m)_{|t=0}.
\end{eqnarray*}
\item Formula for $\partial_{tt}F_2(\lambda_m,0).$
$$
\partial_{tt}F_2(\lambda_m,0)=\frac13\frac{d^3}{dt^3}QG(\lambda_m, tv_m)_{|t=0}+Q\partial_{ff}G(\lambda_m,0)\big[v_m,\widetilde{v}_m\big]
$$
with $\widetilde{v}_m=\frac{d^2}{dt^2}\varphi(\lambda_m, tv_m)_{|t=0}$
and it is given by
$$
\widetilde{v}_m=-\Big[\partial_fG(\lambda_m,0)\Big]^{-1}\frac{d^2}{dt^2} (\textnormal{Id}-Q)G(\lambda_m, tv_m)_{|t=0}
$$
$$
Q\partial_{ff}G(\lambda_m,0)\big[v_m,\widetilde{v}_m\big]=Q\partial_t\partial_sG(\lambda_m, tv_m+s \widetilde{v}_m)_{|t=0,s=0}
$$

\item  Formula for $\partial_t\partial_{\lambda}F_2(\lambda_m,0).$ 
%$$
%\partial_{t}\partial_\lambda\varphi(\lambda_m,0)=-\Big[(\hbox{Id}-Q)D_fG(\lambda_m,0)\Big]^{-1}\partial_{\lambda}\partial_f G(\lambda_m,0)v_m
%$$
\begin{eqnarray*}
\partial_{t}\partial_\lambda F_2(\lambda_m,0)&=&\frac12Q\partial_{\lambda }\partial_{ff}G(\lambda_m,0)[v_m,v_m]+Q\partial_{ff}G(\lambda_m,0)\big[\partial_{\lambda}\partial_g\varphi(\lambda_m,0)v_m,v_m\big]\\
&+&\frac12 Q\partial_{\lambda}\partial_fG(\lambda_m,0) \widetilde{v}_m
\end{eqnarray*}
and 
$$
\partial_{\lambda}\partial_g\varphi(\lambda_m,0)v_m=-\Big[\partial_fG(\lambda_m,0)\Big]^{-1}(\hbox{Id}-Q)\partial_\lambda\partial_f G(\lambda_m,0)v_m.
$$
\item Formula for $\partial_{\lambda\lambda}F_2(\lambda_m,0).$ Assume in addition that $\partial_{\lambda\lambda}G(\lambda_m, 0)=0$, then 
$$
\partial_{\lambda\lambda} F_2(\lambda_m,0)=-2Q\partial_{\lambda}\partial_fG(\lambda_m,0)\Big[\partial_fG(\lambda_m,0)\Big]^{-1}(\textnormal{Id}-Q)\partial_{\lambda}\partial_f G(\lambda_m,0)v_m.
$$

\end{enumerate}
\end{proposition}
\begin{remark}
The loss of transversality which has been  expressed by the identity \eqref{lostra}  should be combined with the formulae of Proposition \ref{platit2} in order to obtain the desired result in the degenerate bifurcation. 

\end{remark}

\begin{proof}

${\bf{(1)}}$ Differentiating $F_2$ given by \eqref{Bifur1} with respect to $\lambda$ gives 
\begin{eqnarray*}
\partial_{\lambda}F_2\big(\lambda_m,0\big)&=&Q \partial_{ff}G\big(\lambda_m,0\big)\big[\partial_\lambda\varphi(\lambda_m,0),\big(v_m+\partial_g\varphi(\lambda_m,0)v_m\big)\big]\\
&+&Q\partial_\lambda\partial_f G\big(\lambda_m,0\big)\big(v_m+\partial_g\varphi(\lambda_m,0)v_m\big)\\
&+&Q\partial_fG\big(\lambda_m,0\big)\partial_{\lambda}\partial_g\varphi(\lambda_m,0)v_m.
\end{eqnarray*}
Using Lemma \ref{platit1} we get
$$
\partial_{\lambda}F_2\big(\lambda_m,0\big)=Q\partial_\lambda \partial_fG\big(\lambda_m,0\big)v_m+Q\partial_fG\big(\lambda_m,0\big)\partial_{\lambda}\partial_g\varphi(\lambda_m,0)v_m.
$$
Combined with \eqref{annul} we entail that 
$$
\partial_{\lambda} F_2\big(\lambda_m,0\big)=Q\partial_\lambda \partial_fG\big(\lambda_m,0\big)v_m.
$$
For the second identity we differentiate \eqref{Bifur1} with respect to $t$, then we get after  straightforward computations
\begin{eqnarray}\label{first1}
\nonumber\partial_t F_2(\lambda_m, t)&=&\int_0^1 s Q\partial_{ff} G\big(\lambda_m, st v_m+\varphi(\lambda_m, st v_m)\big)\big[z_m(t,s), z_m(t,s)\big]ds\\
&+&\int_0^1 s Q\partial_{f} G\big(\lambda_m, st v_m+\varphi(\lambda_m, st v_m)\big)\partial_{gg}\varphi(\lambda_m, st v_m)\big[v_m, v_m\big]ds\
\end{eqnarray}
with the notation
$$
z_m(t,s)\triangleq v_m+(\partial_g\varphi)(\lambda_m, st v_m) v_m.
$$
Using  Lemma \ref{platit1}  we deduce that 
%begin{eqnarray*}
$$2\partial_{t} F_2\big(\lambda_m,0\big)=  Q \partial_{ff}G\big(\lambda_m,0\big)\big[v_m, v_m\big]
+Q\partial_fG\big(\lambda_m,0\big)\partial_{gg}\varphi(\lambda_m,0)\big[v_m,v_m\big] \big).
$$
%\end{eqnarray*}
Thus we obtain by virtue of \eqref{annul}
$$
2\partial_{t} F_2\big(\lambda_m,0\big)=Q \partial_{ff}G\big(\lambda_m,0\big)\big[v_m, v_m\big].
$$
Now it is easy to check   by differentiating the function $t\mapsto QG(\lambda_m, tv_m)$  twice at $t=0$ that
$$
2\partial_{t} F_2\big(\lambda_m,t\big)=\frac{d^2}{dt^2}Q G\big(\lambda_m,tv_m\big)_{|t=0}.
$$
\\
\\
${\bf{(2)}}$ Differentiating  \eqref{first1} with respect to $t$ we get
\begin{eqnarray*}
\partial_{tt} F_2(\lambda_m, t)&=&\int_0^1 s^2 Q\partial_{fff} G\big(\lambda_m,st v_m+\varphi(\lambda_m, st v_m)\big)\big[z_m(t,s),z_m(t,s), z_m(t,s)\big]ds\\
&+&3\int_0^1 s^2 Q\partial_{ff} G\big(\lambda_m,st v_m+\varphi(\lambda_m, st v_m)\big)\Big[z_m(t,s), \partial_{gg}\varphi(\lambda_m, st v_m) \big[v_m,v_m]\Big]ds\\
&+&\int_0^1 s^2 Q\partial_{f} G\big(\lambda_m,st v_m+\varphi(\lambda_m, st v_m)\big) \partial_{ggg}\varphi(\lambda_m, st v_m)\big[v_m, v_m, v_m\big]ds.
\end{eqnarray*}
Using  Lemma \ref{platit1}, \eqref{annul} and \eqref{zero1}  we get
\begin{eqnarray*}
\partial_{tt} F_2(\lambda_m, 0)&=&\frac13\partial_{fff} QG\big(\lambda_m,0\big)\big[v_m,v_m,v_m\big]\\
&+&Q\partial_{ff} G\big(\lambda_m,0\big)\Big[v_m, \partial_{gg}\varphi(\lambda_m,0)\big[v_m,v_m\big]\Big].
\end{eqnarray*}
By setting $\widetilde{v}_m\triangleq \partial_{gg}\varphi(\lambda_m,0)\big[v_m,v_m\big]$ the preceding identity can be written in the form
\begin{eqnarray*}
\partial_{tt} F_2(\lambda_m, 0)&=&\frac13\frac{d^3}{dt^3}QG\big(\lambda_m,tv_m\big)_{|t=0}+Q\partial_{ff} G\big(\lambda_m,0\big)\big[v_m, \widetilde{v}_m\big]\\
&=&\frac13\frac{d^3}{dt^3}QG\big(\lambda_m,tv_m\big)_{|t=0}+\partial_t\partial_sQG\big(\lambda_m,tv_m+s\widetilde{v}_m\big)_{|t=0,s=0}.
\end{eqnarray*}
It remains to compute $\widetilde{v}_m$. Differentiating  the equation \eqref{th11} twice with respect to $t$, we obtain
$$
(\hbox{Id}-Q)\partial_{ff}G(\lambda_m, 0)\big[v_m,v_m\big]+(\hbox{Id}-Q)\partial_f G(\lambda_m, 0)\widetilde{v}_m=0.
$$
Since  $(\hbox{Id}-Q)\partial_f G(\lambda_m, 0): \mathcal{X}_m\to\mathcal{Y}_m$ is invertible then we get in view of \eqref{annul} 
\begin{eqnarray*}
\widetilde{v}_m&=&-\Big[\partial_f G(\lambda_m, 0)\Big]^{-1} (\hbox{Id}-Q)\partial_{ff}G(\lambda_m, 0)\big[v_m,v_m\big]\\
&=&-\Big[\partial_f G(\lambda_m, 0)\Big]^{-1} \frac{d^2}{dt^2}(\hbox{Id}-Q)G(\lambda_m, tv_m)_{|t=0}
\end{eqnarray*}
which is the desired identity for $\partial_{tt}F_2(\lambda_m,0)$.
\\
\\
${\bf{(3)}}$ To compute $\partial_t\partial_\lambda F_2(\lambda_m,0)$ we proceed as before with direct computations combined with Lemma \ref{platit1} 

\begin{eqnarray*}
\partial_t\partial_\lambda F_2(\lambda_m,0)&=&\frac12\partial_\lambda\partial_{ff}QG(\lambda_m,0)[v_m,v_m]+\frac12Q\partial_\lambda\partial_f G(\lambda_m,0)\partial_{gg}\varphi(\lambda_m,0)[v_m,v_m]\\
&+&Q\partial_{ff}G(\lambda_m,0)\big[v_m,\partial_\lambda\partial_g\varphi(\lambda_m,0)v_m\big]\\
&=&\frac12\partial_\lambda\partial_{ff}QG(\lambda_m,0)[v_m,v_m]+\frac12Q\partial_\lambda\partial_f G(\lambda_m,0)\widetilde{v}_m\\
&+&Q\partial_{ff}G(\lambda_m,0)\big[v_m,\partial_\lambda\partial_g\varphi(\lambda_m,0)v_m\big].
\end{eqnarray*}
Differentiating \eqref{th11} respectively with respect to $t$ and $\lambda$ we get
$$
(\hbox{Id}-Q)\partial_{\lambda}\partial_fG(\lambda_m,0)v_m+(\hbox{Id}-Q){\partial_f}G(\lambda_m,0)\partial_{\lambda}\partial_g\varphi(\lambda_m,0)v_m=0.
$$
Since  the restricted operator ${\partial_f}G(\lambda_m,0): \mathcal{X}_m\to \mathcal{Y}_m$ is invertible then
%Recall that we are working without the transversality assumption, that is,
%$$
%Q\partial_{\lambda}\partial_fG(\lambda_m,0)v_m=0.
%$$
%Consequently
\begin{equation}\label{formzy}
\partial_{\lambda}\partial_g\varphi(\lambda_m,0)v_m=-\Big[{\partial_f}G(\lambda_m,0)\Big]^{-1}(\hbox{Id}-Q)\partial_{\lambda}\partial_fG(\lambda_m,0)v_m.
\end{equation}
This achieves the proof of the desired identity.\\
\\
${(\bf{4})}$
Straightforward computations combined with the first part of  Lemma \ref{platit1} yield
\begin{eqnarray*}
\partial_{\lambda\lambda}F_2(\lambda_m,0)&=&Q\partial_{\lambda\lambda}\partial_fG(\lambda_m,0)v_m+2Q\partial_{\lambda}\partial_f G(\lambda_m,0)\partial_{\lambda}\partial_g\varphi(\lambda_m,0)v_m\\
&+&Q\partial_{ff}G(\lambda_m,0)[\partial_{\lambda\lambda}\varphi(\lambda_m,0),v_m]+Q\partial_fG(\lambda_m,0)\partial_{\lambda\lambda}\partial_g\varphi(\lambda_m,0)v_m.
\end{eqnarray*}
Using the second part of Lemma  \ref{platit1} where we need the restriction $\partial_{\lambda\lambda} G(\lambda_m,0)=0$ combined with \eqref{annul} we deduce
\begin{equation*}
\partial_{\lambda\lambda}F_2(\lambda_m,0)=2Q\partial_{\lambda}\partial_f G(\lambda_m,0)\partial_{\lambda}\partial_g\varphi(\lambda_m,0)v_m.
\end{equation*}
It suffices now to use \eqref{formzy} and the desired identity follows.

\end{proof}
\section{Jacobian and Hessian computation }\label{Sec44}
As we have already mentioned in Section \ref{Sec2} the existence of nontrivial solutions to the equation \eqref{V-stat11} in a small neighborhood of  the point $(\lambda_m,0)$ defined in \eqref{lambda6}  is equivalent to solving the finite-dimensional equation \eqref{Bifur1} :
$$
F_2(\lambda,t)=0,\quad (\lambda,t)\quad \hbox{close}\quad (\lambda_m,0).
$$
%The strategy is classical and consists in computing the first and second derivatives of $F_2$ in order to identify the local structure of $F_2$.
We shall be able to decide from the quadratic form associated to $F_2$ whether or not the function $F_2$ admits nontrivial solutions and  prove the Main theorem. We shall establish that  the point  $(\lambda_m,0)$ is a critical point for  $F_2$ and the Hessian is a diagonal matrix  whose eigenvalues change the sign with respect  to the frequency $m.$  Proposition \ref{platit2} stresses   the  complexity of the general formulae for the first and second derivatives and  naturally the computations will be very long. Thus we shall take  special care  in the computations and give  the required details.   
The main result of this section reads as follows.
\begin{proposition}\label{explicitval}
Let $(m,b)\in \widehat{\mathbb{S}}$ which was defined in  \eqref{dim11}. Then the following assertions hold true.
\begin{enumerate}
\item First derivatives: we have
$$
\partial_t F_2(\lambda_m,0)=\partial_\lambda F_2(\lambda_m,0)=0.
$$
\item Expression of $\partial_{\lambda\lambda}F_2(\lambda_m,0).$
$$
\partial_{\lambda\lambda}F_2(\lambda_m,0)= \sqrt{2}m^2b^{1-m} \, \mathbb{W}_m.
$$
\item Expression of $\partial_{t}\partial_{\lambda}F_2(\lambda_m,0).$
$$
\partial_{t}\partial_{\lambda}F_2(\lambda_m,0)=0.
$$
\item Expression of $\partial_{tt}F_2(\lambda_m,0).$

\begin{equation*}
 \partial_{tt}F_2(\lambda_m,0)=\left\{ \begin{array}{ll}
-\frac{\sqrt{2}(1-b^2)^2}{b}\,\mathbb{W}_2, \quad \hbox{if}\quad m=2,\\
\widehat{\alpha}_m\,\mathbb{W}_m, \quad \hbox{if}\quad m\geq3,
\end{array} \right.
\end{equation*}
where
$$
\widehat{\alpha}_m=\frac{1}{\sqrt{2}}\Big[ m(m-1)\frac{(1-b^2)^2}{b}+{\widehat{\beta}_m \widehat{\gamma}_m}\Big]
$$

with
$$
\widehat{\beta}_m =\frac{2m(b^2+b^{m})^2}{ (b^m+1)^2(-b^{2m}+2b^m+1)}
$$
and 
$$\widehat{\gamma}_m=(m-2)b+mb^{2m-1}+(4m-6)b^{m+1}
+4(m-1)b^{2m+1}+2mb^{3m-1}.
$$

\end{enumerate}

\end{proposition}
The proof is quite involved and falls naturally  into several steps.

%%
%%%
\subsection{Preparation}
The basic tools and notation that we shall deal with in this part were previously introduced in Section \ref{Sec2} and for the clarity we refer the reader to this section. 
Further  notation will also be needed :
\begin{equation}\label{Notqy1}
b_1=1, b_2=b, \alpha_1=\varepsilon b, \alpha_2=1.
\end{equation}
By introducing the functions :
\begin{equation}\label{phi_9}
\phi_j(t,w)=b_j w+t\alpha_j \overline{w}^{m-1}.
\end{equation}
one can rewrite  the functions of  \eqref{G_j} as follows\begin{equation}\label{G59}
G_j(\lambda_m, tv_m)=\hbox{Im}\Bigg\{\Big[(1-\lambda_m)\overline{\phi_j(t,w)}+I(\phi_j(w))\Big] w\Big(b_j+t(1-m)\alpha_j \overline{w}^m\Big)\Bigg\},
\end{equation}
with
$$
I(\phi_j(t,w))=I_1(\phi_j(t,w))-I_2(\phi_j(t,w))
$$
and 
$$ I_i(\phi_j(t,w))=\fint_{\mathbb{T}}
\frac{\overline{\phi_j(t,w)}-\overline{\phi_i(t,\tau)}}{\phi_j(t,w)-\phi_i(t,\tau)}\phi'_i(t,\tau)d\tau.
$$
According to Proposition \ref{platit2} the quantities $\frac{d^k}{dt^k} G(\lambda_m, tv_m)_{|t=0}$  for $k=2,3$ play central role on the structure of  Taylor expansion of $F_2$ at the second order.  Their calculations are fairly complicated and as we shall see throughout  the next sections 
we will be able to  give explicit expressions.
\subsection{Jacobian computation}
In this section we aim to  prove the first point $(1)$ of Proposition \ref{explicitval} which merely says that the point $(\lambda_m,0)$ is a critical point for the function $F_2$ defined in \eqref{Bifur1}. In other words,
$$
\partial_\lambda F_2(\lambda_m,0)=\partial_t F_2(\lambda_m,0)=0.
$$
The first one is an immediate consequence  of Proposition \ref{platit2}-$(1)$ combined with the loss of transversality written in the form \eqref{lostra}.  In regard to  the second identity we shall use once again   Proposition  \ref{platit2}-$(1)$
$$
\partial_t F_2(\lambda_m,0)=\frac12\frac{d^2}{dt^2}QG(\lambda_m, tv_m)_{|t=0}.
$$
The basic ingredient is the following identity whose  explicit formula will be proved and used again later,
\begin{eqnarray}\label{Sec-first}
\nonumber \frac{d^2}{dt^2} G(\lambda_m, tv_m)_{|t=0}&=&2m\left( \begin{array}{c}
\big( b^2-\varepsilon b^{m}\big)^2 \\
0
\end{array} \right) \, e_{2m}\\
&\triangleq&\left( \begin{array}{c}
 \widehat{\alpha} \\
0
\end{array} \right)e_{2m}.
\end{eqnarray}
Assume for a while this expression, then  we get by means of  the structure of the projection $Q$ introduced in \eqref{project1}
\begin{equation}\label{tit45}
\frac{d^2}{dt^2} Q G(\lambda_m, tv_m)_{|t=0}=0
\end{equation}
and consequently
$$
\partial_tF_2(\lambda_m,0)=0.
$$
Now we shall turn to the proof of \eqref{Sec-first}. Before proving that it is useful to introduce   some  notation. Set
\begin{equation}\label{TQ1}
A=b_j w-b_i\tau, \quad B=\alpha_j\overline{w}^{m-1}-\alpha_i\overline{\tau}^{m-1}\quad C=\alpha_i(1-m)\overline{\tau}^{m}
\end{equation}
then we may write 
$$
I_i(\phi_j(t,w))=\fint_{\mathbb{T}}\frac{\overline{A}+ t\overline{B}}{A+tB}\big(b_i+ tC\big) d\tau,
$$
with $\phi_j$ being the function already defined in \eqref{phi_9}.
 It is not so hard to check that
\begin{eqnarray}\label{dzr22}
\nonumber\frac{d^2}{dt^2} G_j(\lambda_m, tv_m)_{|t=0}&=&\hbox{Im}\Big\{2(1-\lambda_m)\alpha_j^2(1-m)+b_j w\frac{d^2}{dt^2} I(\phi_j(t))_{|t=0}\\
\nonumber &+&2(1-m)\alpha_j \overline{w}^{m-1}\frac{d}{dt} I(\phi_j(t))_{|t=0}\Big\}\\
&=&\hbox{Im}\Big\{b_j w\frac{d^2}{dt^2} I(\phi_j(t))_{|t=0}+2(1-m)\alpha_j \overline{w}^{m-1}\frac{d}{dt} I(\phi_j(t))_{|t=0}\Big\}.
\end{eqnarray}
The next task is  to compute each term appearing in the right-hand side of  the preceding identity.\\
\\
{\bf Computation of $\frac{d}{dt} I(\phi_j(t))_{|t=0}.$} By the usual rules of the derivation we can check easily that
$$
\frac{d}{dt}(I_i(\phi_j(t,\cdot))_{|t=0} =\fint_{\mathbb{T}}\frac{\overline{A} }{A^2}\Big( AC-b_i B\Big) d\tau+b_i\fint_{\mathbb{T}}\frac{\overline{B}}{A}d\tau.
$$
It is convenient at this stage to make some scaling arguments which are frequently used later.
By homogeneity argument that we can implement by change of variables we find  real constants $\mu_{ij},\gamma_{ij}$ such that
$$
\fint_{\mathbb{T}}\frac{\overline{B}}{A}d\tau=\mu_{i,j} w^{m-1}
$$
and
$$
\fint_{\mathbb{T}}\frac{\overline{A} }{A^2}\Big( AC-b_i B\Big) d\tau=\gamma_{ij} \overline{w}^{m+1},
$$
with the expression
$$
\gamma_{ij}=\alpha_i(1-m)\fint_{\mathbb{T}}\frac{b_j-b_i\overline{\tau}}{b_j -b_i\tau}\overline{\tau}^m d\tau-b_i\fint_{\mathbb{T}}\frac{b_j-b_i\overline{\tau}}{(b_j -b_i\tau)^2}\big(\alpha_j-\alpha_i\overline{\tau}^{m-1}\big) d\tau.
$$
Coming back to  the notation \eqref{Notqy1} we may write
$$
\gamma_{12}=\alpha_1(1-m)\fint_{\mathbb{T}}\frac{b-\overline{\tau}}{b -\tau}\overline{\tau}^m d\tau-b_1\fint_{\mathbb{T}}\frac{b-\overline{\tau}}{(b -\tau)^2}\big(\alpha_2-\alpha_1\overline{\tau}^{m-1}\big) d\tau.
$$
Since  the integrands decay quickly more than $\frac{1}{\tau^2}$ the applying   residue theorem at $\infty$ we get,
\begin{equation}\label{merc1}
\gamma_{12}=0.
\end{equation}
Again with \eqref{Notqy1} we write 
$$
\gamma_{21}=(1-m)\fint_{\mathbb{T}}\frac{1-b\overline{\tau}}{1-b\tau}\overline{\tau}^m d\tau-b\fint_{\mathbb{T}}\frac{1-b\overline{\tau}}{(1-b\tau)^2}\big(\varepsilon b-\overline{\tau}^{m-1}\big) d\tau.
$$
Next we shall make use of the following  identities whose proofs  can be easily obtained through  residue theorem : For $m\in \NN^\star$
\begin{eqnarray}\label{form76}
\nonumber \fint_{\mathbb{T}}\frac{\overline{\tau}^m}{1-b\tau}d\tau&=&\fint_{\mathbb{T}}\frac{{\tau}^{m-1}}{\tau-b}d\tau\\
&=& b^{m-1}
\end{eqnarray}
and 
\begin{eqnarray}\label{form706}
\nonumber  \fint_{\mathbb{T}}\frac{\overline{\tau}^m}{(1-b\tau)^2}d\tau&=& \fint_{\mathbb{T}}\frac{{\tau}^{m}}{(\tau-b)^2}d\tau\\
 &=& m\,b^{m-1}.
\end{eqnarray}
Consequently we find
\begin{eqnarray}\label{merc2}
\nonumber \gamma_{21}&=&(1-m)\big(b^{m-1}- b^{m+1}\big)-b\big(-(m-1) b^{m-2}-\varepsilon b^2+m b^{m}\big)\\
&=&\varepsilon b^3-b^{m+1}.
\end{eqnarray}
For the diagonal terms $i=j$, we write by the definition
$$
\gamma_{ii}=\alpha_i(1-m)\fint_{\mathbb{T}}\frac{1-\overline{\tau}}{1-\tau}\overline{\tau}^m d\tau-\alpha_i\fint_{\mathbb{T}}\frac{1-\overline{\tau}}{(1-\tau)^2}\big(1-\overline{\tau}^{m-1}\big) d\tau.
$$
Observe that the  integrands are extended holomorphically to $\CC^\star$ and therefore by using residue theorem at $\infty$ we get
\begin{equation}\label{merc3}
\gamma_{ii}=0.
\end{equation}
Putting together  \eqref{merc1},\eqref{merc2} and \eqref{merc3} we infer
\begin{eqnarray}\label{merc5}
\nonumber\frac{d}{dt}I(\phi_1(t,w))_{|t=0}&=&\frac{d}{dt}I_1(\phi_1(t,w))_{|t=0}-\frac{d}{dt}I_2(\phi_1(t,w))_{|t=0}\\
\nonumber&=&\big(\gamma_{11}-\gamma_{21}\big)\overline{w}^{m+1}+(\mu_{11}-b\mu_{21}) w^{m-1}\\
&=&(b^{m+1}-\varepsilon b^3)\overline{w}^{m+1}+(\mu_{11}-b\mu_{21}) w^{m-1}.
\end{eqnarray}
Proceeding in a similar way to \eqref{merc5}   we obtain
\begin{eqnarray}\label{merc45}
\nonumber\frac{d}{dt}I(\phi_2(t,w))_{|t=0}&=&\frac{d}{dt}I_1(\phi_2(t,w))_{|t=0}-\frac{d}{dt}I_2(\phi_2(t,w))_{|t=0}\\
\nonumber&=&\big(\gamma_{12}-\gamma_{22}\big)\overline{w}^{m+1}+(\mu_{12}-b\mu_{22}) w^{m-1}\\
&=&(\mu_{12}-b\mu_{22}) w^{m-1}.
\end{eqnarray}
\\
{\bf Computation of $\frac{d^2}{dt^2} I(\phi_j(t))_{|t=0}.$} Straightforward computations yield
$$
\frac{d^2}{dt^2}(I_i(\phi_j(t,\cdot))_{|t=0} )=2\fint_{\mathbb{T}}\frac{A\overline{B}-\overline{A} B}{A^3}\Big( AC-b_i B\Big) d\tau.
$$
By scaling one  finds real constants $\widehat\mu_{i,j}$ and $\eta_{i,j}$ such that
$$
\fint_{\mathbb{T}}\frac{\overline{B}}{A^2}\Big( AC-b_i B\Big) d\tau
=\widehat\mu_{i,j}\overline{w}, \quad \fint_{\mathbb{T}}\frac{\overline{A} B}{A^3}\Big( b_i B-AC\Big) d\tau
=\eta_{i,j}\overline{w}^{2m+1},
$$
with 
$$
\widehat\mu_{i,j}=\fint_{\mathbb{T}}\frac{\alpha_j -\alpha_i\tau^{m-1}}{(b_j-b_i\tau)^2}\Big(\alpha_i(1-m)(b_j-b_i\tau)\overline{\tau}^m-b_i(\alpha_j-\alpha_i\overline{\tau}^{m-1})\Big)d\tau.
$$
Hence we deduce
$$
\frac{d^2}{dt^2}(I_i(\phi_j(t,\cdot))_{|t=0} =2\widehat\mu_{i,j}\overline{w}+2\eta_{i,j}\overline{w}^{2m+1}.
$$
Thus we obtain successively 
$$
\frac{d^2}{dt^2}(I(\phi_1(t,\cdot))_{|t=0} =2(\widehat\mu_{1,1}-\widehat\mu_{2,1})\overline{w}+2(\eta_{1,1}-\eta_{2,1})\overline{w}^{2m+1}
$$
$$
\frac{d^2}{dt^2}(I(\phi_2(t,\cdot))_{|t=0} =2(\widehat\mu_{1,2}-\widehat\mu_{2,2})\overline{w}+2(\eta_{1,2}-\eta_{2,2})\overline{w}^{2m+1}.
$$
We are left with the task of determining  all the involved coefficients. We start with computing  $\widehat\mu_{i,i}$ which follows easily from residue theorem at $\infty$
\begin{eqnarray*}
\widehat\mu_{i,i}&=&\frac{\alpha_i^2}{b_i}\fint_{\mathbb{T}}\frac{1-\tau^{m-1}}{(1-\tau)^2} \Big((1-m)(1-\tau)\overline{\tau}^m-(1-\overline{\tau}^{m-1})\Big)d\tau\\
&=&-\frac{\alpha_i^2}{b_i}\fint_{\mathbb{T}}\frac{1-\tau^{m-1}}{(1-\tau)^2} d\tau\\
&=&(m-1)\frac{\alpha_i^2}{b_i}.
\end{eqnarray*}
This  implies
\begin{equation}\label{Bar2}
\widehat\mu_{1,1}=(m-1) b^2,\quad \widehat\mu_{2,2}=(m-1)\frac{1}{b}\cdot
\end{equation}
For the coefficient  $\widehat\mu_{1,2}$ we may write  in view of \eqref{form706}
\begin{eqnarray}\label{Barr1}
\nonumber\widehat\mu_{1,2}&=&\fint_{\mathbb{T}}\frac{1-\varepsilon b\tau^{m-1}}{(b-\tau)^2} \Big(\varepsilon b(1-m)(b-\tau)\overline{\tau}^m-(1-\varepsilon b\overline{\tau}^{m-1})\Big)d\tau\\
\nonumber&=&-\fint_{\mathbb{T}}\frac{1-\varepsilon b\tau^{m-1}}{(b-\tau)^2} d\tau\\
&=&\varepsilon(m-1)b^{m-1}.
\end{eqnarray}
As to  the coefficient $\widehat\mu_{2,1}$ we apply once again the residue theorem,
\begin{eqnarray*}
\widehat\mu_{2,1}&=&\fint_{\mathbb{T}}\frac{\varepsilon b-\tau^{m-1}}{(1-b\tau)^2} \Big((1-m)(1-b\tau)\overline{\tau}^m-b(\varepsilon b-\overline{\tau}^{m-1})\Big)d\tau\\
&=&\fint_{\mathbb{T}}\frac{\varepsilon b-\tau^{m-1}}{(1-b\tau)^2} \Big((1-m)\overline{\tau}^m+bm\overline{\tau}^{m-1})\Big)d\tau\\
&=&\fint_{\mathbb{T}}\frac{\varepsilon b(1-m)\overline{\tau}^m+\varepsilon b^2 m\overline{\tau}^{m-1}+(m-1)\overline{\tau}}{(1-b\tau)^2}d\tau.
\end{eqnarray*}
Using  \eqref{form706}   and after some   cancellations we get
\begin{equation}\label{Bar3}
\widehat\mu_{2,1}=(m-1) .
\end{equation}
Putting together \eqref{Bar2},\eqref{Barr1} and \eqref{Bar3} we infer
\begin{eqnarray}\label{sd12}
\frac{d^2}{dt^2}(I(\phi_1(t,\cdot))_{|t=0} &=&2(m-1)\big( b^2-1\big)\overline{w}+2(\eta_{1,1}-\eta_{2,1})\overline{w}^{2m+1},\\
\nonumber\frac{d^2}{dt^2}(I(\phi_2(t,\cdot))_{|t=0} &=&2(m-1)\big(\varepsilon b^{m-1}-b^{-1}\big)\overline{w}+2(\eta_{1,2}-\eta_{2,2})\overline{w}^{2m+1}.
\end{eqnarray}
Now we shall move to  the computation of the terms $\eta_{i,j}$ which are defined by
$$
\eta_{i,j}=\alpha_i(m-1)\fint_{\mathbb{T}}\frac{(b_j-b_i\overline{\tau})(\alpha_j-\alpha_i\overline{\tau}^{m-1})}{(b_j-b_i\tau)^2}\overline{\tau}^{m}d\tau+b_i\fint_{\mathbb{T}}\frac{(b_j-b_i\overline{\tau})(\alpha_j-\alpha_i\overline{\tau}^{m-1})^2}{(b_j-b_i\tau)^3}d\tau.
$$
We readily get by invoking  residue theorem at $\infty$ 
\begin{eqnarray*}
\eta_{1,2}&=&\alpha_1(m-1)\fint_{\mathbb{T}}\frac{(b-\overline{\tau})(\alpha_2-\alpha_1\overline{\tau}^{m-1})}{(b-\tau)^2}\overline{\tau}^{m}d\tau+\fint_{\mathbb{T}}\frac{(b-\overline{\tau})(\alpha_2-\alpha_1\overline{\tau}^{m-1})^2}{(b-\tau)^3}d\tau\\
&=&0.
\end{eqnarray*}
Always with  the notation \eqref{Notqy1} we can write
\begin{eqnarray*}
\eta_{2,1}&=&(m-1)\fint_{\mathbb{T}}\frac{(1-b\overline{\tau})(\alpha_1-\overline{\tau}^{m-1})}{(1-b\tau)^2}\overline{\tau}^{m}d\tau+b\fint_{\mathbb{T}}\frac{(1-b\overline{\tau})(\alpha_1-\overline{\tau}^{m-1})^2}{(1-b\tau)^3}d\tau.
\end{eqnarray*}
Developing the numerator and using  the identity  \eqref{form706} lead to
\begin{eqnarray*}
\fint_{\mathbb{T}}\frac{(1-b\overline{\tau})(\alpha_1-\overline{\tau}^{m-1})}{(1-b\tau)^2}\overline{\tau}^{m}d\tau&=&\fint_{\mathbb{T}}\frac{\alpha_1\overline{\tau}^{m}-\overline{\tau}^{2m-1}-b\alpha_1\overline{\tau}^{m+1}+b\overline{\tau}^{2m})}{(1-b\tau)^2}d\tau\\
\\
&=&\varepsilon m b^m-(2m-1) b^{2m-2}-\varepsilon (m+1) b^{m+2}+2m b^{2m}.
\end{eqnarray*}
Concerning the last integral term of $\eta_{2,1}$ it can be calculated as follows,
\begin{eqnarray*}
\fint_{\mathbb{T}}\frac{(1-b\overline{\tau})(\alpha_1-\overline{\tau}^{m-1})^2}{(1-b\tau)^3}d\tau
&=&\fint_{\mathbb{T}}\frac{-2\alpha_1\overline{\tau}^{m-1}+\overline{\tau}^{2m-2}-b\alpha_1^2\overline{\tau}+2\alpha_1 b\overline{\tau}^m-b\overline{\tau}^{2m-1}}{(1-b\tau)^3}d\tau\\
\\
&=&-\varepsilon m(m-1) b^{m-1}+(2m-1) (m-1) b^{2m-3}\\
\\
&-&b^3+\varepsilon m (m+1) b^{m+1}-m (2m-1) b^{2m-1}.
\end{eqnarray*}
Notice that  we have used the following  identity which can be easily derived from residue theorem,
\begin{eqnarray}\label{Idgh1}
\nonumber\fint_{\mathbb{T}}\frac{\overline{\tau}^m}{(1-b\tau)^3}d\tau&=&\fint_{\mathbb{T}}\frac{{\tau}^{m+1}}{(\tau-b)^3}d\tau\\
&=&\frac12(m+1) m b^{m-1}\cdot
\end{eqnarray}
Putting together the preceding formulae yields
\begin{equation}\label{tre11}
\eta_{2,1}=\varepsilon (m+1) b^{m+2}-m b^{2m}-b^4.
\end{equation}
As to  the diagonal terms  $\eta_{i,i}$, we write 
in view of  residue theorem at $\infty$
\begin{eqnarray}\label{th117}
\nonumber\eta_{i,i}&=&(m-1)\frac{\alpha_i^2}{b_i}\fint_{\mathbb{T}}\frac{(1-\overline{\tau})(1-\overline{\tau}^{m-1})}{(1-\tau)^2}\overline{\tau}^{m}d\tau+\frac{\alpha_i^2}{b_i}\fint_{\mathbb{T}}\frac{(1-\overline{\tau})(1-\overline{\tau}^{m-1})^2}{(1-\tau)^3}d\tau\\
&=&0.
\end{eqnarray}
Combining  \eqref{sd12}, \eqref{Idgh1}, \eqref{tre11} and \eqref{th117} we come to the following expressions
\begin{eqnarray}\label{ds13}
\frac12\frac{d^2}{dt^2}(I(\phi_1(t,w))_{|t=0} &=&(m-1)\big( b^2-1\big)\overline{w}+\Big(b^4+m b^{2m}-\varepsilon (m+1) b^{m+2}\Big)\overline{w}^{2m+1}\\
\nonumber\frac12\frac{d^2}{dt^2}(I(\phi_2(t,w))_{|t=0} &=&(m-1)\big(\varepsilon b^{m-1}-b^{-1}\big)\overline{w}.
\end{eqnarray}
Now plugging \eqref{merc5},\eqref{merc45} and \eqref{ds13} into \eqref{dzr22} we deduce  that
\begin{eqnarray*}
\frac{d^2}{dt^2} G_1(\lambda_m, tv_m)_{|t=0}&=&2m\Big(b^4-2\varepsilon b^{m+2}+b^{2m}\Big) e_{2m}\\
\frac{d^2}{dt^2} G_2(\lambda_m, tv_m)_{|t=0}&=& 0,
\end{eqnarray*}
with the usual  notation
$e_m\triangleq\hbox{Im}\big(\overline{w}^m\big)$. This completes the proof of the identity \eqref{Sec-first}.

\subsection{Hessian computation}
The main object is to  prove the points $(2), (3)$ and $(4)$ of \mbox{Proposition \ref{explicitval}.} This will be done separately in several sections  and we shall start with the second point.
\subsubsection{Computation of $\partial_{\lambda\lambda}F_2(\lambda_m,0)$}
According to  Proposition \ref{platit2} one has the formula
\begin{eqnarray*}
\partial_{\lambda\lambda} F_2(\lambda_m,0)&=&2Q\partial_{\lambda}\partial_fG(\lambda_m,0)\partial_\lambda\partial_g\varphi(\lambda_m,0)v_m
\\
&=&-2Q\partial_{\lambda}\partial_fG(\lambda_m,0)\Big[\partial_fG(\lambda_m,0)\Big]^{-1}\partial_{\lambda}\partial_f G(\lambda_m,0)v_m.
\end{eqnarray*}
We point out that we have used in the last line the identity \eqref{lostra} illustrating  the loss of transversality and which implies that
$$
(\hbox{Id}-Q)\partial_{\lambda}\partial_f G(\lambda_m,0)v_m=\partial_{\lambda}\partial_f G(\lambda_m,0)v_m.
$$
As we are dealing with smooth functions, Fr\'echet derivative coincides with G\^ateaux derivative and furthermore one can  use Schwarz theorem combined with \eqref{ddff},
\begin{eqnarray*}
\partial_{\lambda}\partial_f G(\lambda_m,0)v_m&=&\big\{\partial_{\lambda}\mathcal{L}_{\lambda,b}v_m\big\}_{|\lambda=\lambda_m}\\
&=& m b\left( \begin{array}{c}
  \varepsilon \\
1
\end{array} \right) \, e_m.
\end{eqnarray*}
Using \eqref{fs1}, \eqref{fs10} and \eqref{Inv} we obtain 
\begin{eqnarray}\label{zqs11}
\nonumber\partial_\lambda\partial_g\varphi(\lambda_m,0)v_m&=&-mb \Big[\partial_fG(\lambda_m,0)\Big]^{-1} \Big\{\left( \begin{array}{c}
  \varepsilon \\
1
\end{array} \right) \, e_m\Big\}\\
&=& mb^{1-m}\left( \begin{array}{c}
 1\\
0
\end{array} \right) \,\overline{w}^{m-1}.
\end{eqnarray}
By invoking  \eqref{ddff} we can  infer that

\begin{eqnarray}\label{zqs110}
\nonumber \partial_{\lambda}\partial_fG(\lambda_m,0)\partial_\lambda\partial_g\varphi(\lambda_m,0)v_m &=&mb^{1-m}\begin{pmatrix}
m & 0 \\
0& bm
\end{pmatrix}\left( \begin{array}{c}
 1\\
0
\end{array} \right)\, e_m\\
&=&m^2 b^{1-m}\left( \begin{array}{c}
 1 \\
0
\end{array} \right) \, e_m.
\end{eqnarray}
It follows  from \eqref{project1} that 
\begin{eqnarray}\label{zqs120}
\nonumber\partial_{\lambda\lambda} F_2(\lambda_m,0)
\nonumber&=&2m^2b^{1-m}\Big\langle \mathbb{W}, \left( \begin{array}{c}
 1\\
0
\end{array} \right)\Big\rangle\,\mathbb{W}_m\\
&=&\sqrt{2}m^2b^{1-m}\mathbb{W}_m.
\end{eqnarray}
This finishes the proof of Proposition \ref{explicitval}-$(2)$.

\subsubsection{Computation of $\partial_{t}\partial_{\lambda}F_2(\lambda_m,0)$} We intend to prove that this mixed derivative vanishes on the critical point. This will be done in several steps and the first one is to use the following formula stated in Proposition  \ref{platit2}-$(3)$ :
\begin{eqnarray*}
\partial_{t}\partial_\lambda F_2(\lambda_m,0)&=&\frac12Q\partial_{\lambda }\partial_{ff}G(\lambda_m,0)[v_m,v_m]+Q\partial_{ff}G(\lambda_m,0)\big[\partial_{\lambda}\partial_g\varphi(\lambda_m,0)v_m,v_m\big]\\
&+&\frac12 Q\partial_{\lambda}\partial_fG(\lambda_m,0) \widetilde{v}_m.
\end{eqnarray*}
The easiest term is the first one and direct computations  yield,
\begin{eqnarray*}
\partial_\lambda \partial_{ff}G_j(\lambda_m,0)[v_m,v_m]&=&\frac{d^2}{dt^2}\partial_\lambda G_j(\lambda, tv_m)_{|\lambda=\lambda_m, t=0}\\
&=&-\frac{d^2}{dt^2}\hbox{Im} \Big\{ w\overline{\phi_j(t,w)}\phi_j^\prime(t,w)\Big\}\\
&=&0,
\end{eqnarray*}
with $\phi_j$ being the function defined in \eqref{phi_9}. However, the second  second term is much more complicated and one notices first that
\begin{eqnarray*}
Q\partial_{ff}G(\lambda_m,0)\big[\partial_{\lambda}\partial_g\varphi(\lambda_m,0)v_m,v_m\big]&=&Q\partial_s\partial_tG\big(\lambda_m, tv_m+s \partial_{\lambda}\partial_g\varphi(\lambda_m,0)v_m\big)_{|t,s=0}.
\end{eqnarray*}
It suffices now to combine  Lemma \ref{lemexp} below  with \eqref{zqs11} in order  to  deduce that
$$
Q\partial_s\partial_tG(\lambda_m, tv_m+s \partial_{\lambda}\partial_g\varphi(\lambda_m,0)v_m)_{|t,s=0}
=0.
$$
With regard to the last term  we invoke  once again Lemma \ref{lemexp} which implies that
$$
Q\partial_{\lambda}\partial_fG(\lambda_m,0) \widetilde{v}_m=0.
$$
Finally, after gathering the preceding results  we obtain the desired result, namely,
$$
\partial_{t}\partial_\lambda F_2(\lambda_m,0)=0.
$$
To end  this section we shall prove the next result which has been used in the foregoing computations.
\begin{lemma}\label{lemexp}
Let $(m,b)\in \widehat{\mathbb{S}}$, then  the following holds true.
\begin{enumerate}
\item Expression of $\widetilde{v}_m:$

\begin{equation*}
\widetilde{v}_2=0\quad\hbox{and}\quad \widetilde{v}_m(w)=\widehat{\beta}_m\left( \begin{array}{c}
-1-2b^m \\
b^{2m-1}
\end{array} \right) \overline{w}^{2m-1}, \quad m\geq3.
\end{equation*}
with
$$
\widehat{\beta}_m\triangleq\frac{2m(b^2+b^{m})^2}{ (b^m+1)^2(-b^{2m}+2b^m+1)}\cdot
$$
Furthermore, 
$$
Q\partial_\lambda\partial_f G(\lambda_m,0) \widetilde{v}_m=0.
$$
\item Let ${Z}_m=\mathbb{\alpha}\overline{w}^{m-1}$ and $\widehat{Z}_m= \mathbb{\beta}\overline{w}^{m-1}$, with $\mathbb{\alpha}, \mathbb{\beta}\in \RR^2.$ Then 
$$
Q\partial_t\partial_sG(\lambda_m, t\, Z_m+s\, \widehat{Z}_m)_{|t,s=0}=0.
$$
\end{enumerate}
\end{lemma}
\begin{proof}

${\bf(1)}$ The definition of $\widetilde{v}_m$ was given in Proposition \ref{platit2} and therefore by making appeal \mbox{to   \eqref{Inv}} and \eqref{Sec-first} we get
\begin{eqnarray}\label{stri1}
\nonumber\widetilde{v}_m&=&-\Big[\partial_fG(\lambda_m,0)\Big]^{-1}\frac{d^2}{dt^2} (Id-Q)G(\lambda_m, tv_m)_{|t=0}\\
\nonumber&=&-\Big[\partial_fG(\lambda_m,0)\Big]^{-1}\left( \begin{array}{c}
 \widehat{\alpha} \\
0
\end{array} \right)e_{2m}\\
&=&-M_{2m}^{-1}\left( \begin{array}{c}
 \widehat{\alpha} \\
0
\end{array} \right)\overline{w}^{2m-1}.
\end{eqnarray}
The frequency mode $m=2$ is very typical. Indeed, we recall that $ \widehat{\alpha}=2m(b^2-\varepsilon b^m)^2$ and  thus we deduce from the definition of   $\varepsilon$ in \eqref{formm} that  for $m=2$
$$
\widehat{\alpha}=0
$$
which in turn  implies that
$$
\widetilde{v}_2=0.
$$
Now we move to higher frequencies $m\geq3$ which requires more computations. From the preceding formula for $\widetilde{v}_m$ one needs to compute the inverse of the matrix $M_{2m}.$ This matrix was defined in   \eqref{Matr1} and  takes the form
\begin{eqnarray*}
 M_{2m}(\lambda_m)&=&\begin{pmatrix}
2m\lambda_m -1-2m b^2 & b^{2m+1} \\
 -b^{2m}& b\big(2m\lambda_m-2m+1\big)
\end{pmatrix}.
\end{eqnarray*}
Note that we are in the case  $b=b_m$ which was defined in \eqref{relat1} and for 
not to burden the notation we drop the subscript $m$. Also, recall that $\lambda_m=\frac{1+b^2}{2}$ and therefore using the \mbox{relation \eqref{relat1}} we find that
\begin{eqnarray*}
 M_{2m}(\lambda_m)&=&\begin{pmatrix}
1+2b^m& b^{2m+1} \\
 -b^{2m}&-b-2b^{m+1}
\end{pmatrix}.
\end{eqnarray*}
Its inverse can be computed in usual way and  it comes that
%\begin{eqnarray}\label{Matr1}
%\nonumber M_{2m}^{-1}&=&\frac{1}{\hbox{det} M_{2m}}\begin{pmatrix}
%b(2m\lambda-2m+1) & -b^{2m+1} \\
% b^{2m}&2m\lambda-1-2m b^2\end{pmatrix}
%\end{eqnarray}
%Since $\lambda=\lambda_m=\frac{1+b^2}{2}$ then we get  from the relation \eqref{relat1}
\begin{eqnarray*}%\label{Matr1}
\nonumber M_{2m}^{-1}(\lambda_m)&=&\frac{1}{\hbox{det } M_{2m}(\lambda_m)}\begin{pmatrix}
-b-2b^{m+1} & -b^{2m+1} \\
 b^{2m}&1+2 b^m\end{pmatrix}.
\end{eqnarray*}
If we insert  this expression in \eqref{stri1}  we easily obtain
\begin{eqnarray*}
\widetilde{v}_m&=&-\frac{ b\widehat{\alpha}}{ \hbox{det}M_{2m}}\left( \begin{array}{c}
-1-2b^m\\
b^{2m-1}
\end{array} \right) \overline{w}^{2m-1}.
\end{eqnarray*}
As to the determinant  of $M_{2m}$ it can be  factorized  in the form,
$$
\hbox{det } M_{2m}=b(b^m+1)^2(b^{2m}-2b^m-1).
$$
It follows from th expression of $\widetilde\alpha$ stated in \eqref{Sec-first}  that
\begin{equation}\label{beta11}
\widetilde{v}_m=\widehat{\beta}_m\left( \begin{array}{c}
-1-2b^m \\
b^{2m-1}
\end{array} \right) \overline{w}^{2m-1},\quad \widehat{\beta}_m =\frac{2m(b^2+b^{m})^2}{ (b^m+1)^2(-b^{2m}+2b^m+1)}, \quad m\geq3.
\end{equation}
%Using \eqref{Inv} we obtain
%$$
%\widetilde{v}_m=-[M_{2m}^{-1}B_2]\, \overline{w}^{2m-1}:=A_2\overline{w}^{2m-1}:=(a_1,a_2)\overline{w}^{2m-1}.
%$$
Now applying \eqref{ddff} we find the following structure
$$
\partial_\lambda\partial_fG(\lambda_m,0)\widetilde{v}_m=C_2\, e_{2m},\quad C_2\in\RR^2.
$$
Consequently we obtain due to \eqref{project1}
$$
Q\partial_\lambda\partial_fG(\lambda_m,0)\widetilde{v}_m=0.
$$
Actually, at this stage we have not used the explicit form of $\widetilde{v}_m$. However, the exact  form will be crucial in the computation of  $\partial_{tt}F_2(\lambda_m,0)$ that will be done later  in \mbox{Subsection \ref{SectionQ}.}
\\
\\
%%
%%${\bf{(2)}}$ 
%%
${\bf{(2)}}$ Set  $Z_m=\alpha \overline{w}^{m-1}$ and $\widehat{Z}_m=\beta \overline{w}^{m-1}$  with $\alpha=(\alpha_1,\alpha_2)$ and $  \beta=(\beta_1,\beta_2)\in \RR^2.$ To simplify the expressions we need to fix some notation :
\begin{equation}\label{phi_x}
\phi_j(t,s,w)=b_j w+(t\alpha_j +s\beta_j)\overline{w}^{m-1}.
\end{equation}
Then from \eqref{G_j} and \eqref{kernexp} we may write
$$
G_j(\lambda_m, tZ_m+s\widehat{Z}_m)=\hbox{Im}\Bigg\{\Big[(1-\lambda_m)\overline{\phi_j(t,s,w)}+I(\phi_j(t,s,w))\Big] w\Big(b_j+(t\alpha_j+s\beta_j)(1-m)\overline{w}^m\Big)\Bigg\},
$$
with
$$
I(\phi_j(t,s,w))=I_1(\phi_j(t,s,w))-I_2(\phi_j(t,s,w)).
$$
and
$$ I_i(\phi_j(t,s,w))=\fint_{\mathbb{T}}
\frac{\overline{\phi_j(t,s,w)}-\overline{\phi_i(t,s,\tau)}}{\phi_j(t,s,w)-\phi_i(t,s,\tau)}\phi'_i(t,s,\tau)d\tau.
$$
It is plain according to our notation  that
$$
 I_i(\phi_j(t,s,w))=\fint_{\mathbb{T}}\frac{K(\overline\tau,\overline w)}{K(\tau,w)}{\Big(b_i+[t\alpha_i+s\beta_i](1-m)\overline{\tau}^m\Big)} d\tau,
$$
with
$$
K(\tau,w)\triangleq b_j w-b_i\tau+\big(t\alpha_j+ s\beta_j\big) \overline{w}^{m-1}-\big(t\alpha_i+s\beta_i\big)\overline{\tau}^{m-1}.
$$
Differentiating with respect to $t$  one readily obtains
\begin{eqnarray*}
\partial_tI_i(\phi_j(t,s,w)_{|t=s=0}&=&b_i\fint_{\mathbb{T}}\frac{\alpha_j w^{m-1}-\alpha_i\tau^{m-1}}{b_j w-b_i\tau}d\tau+\alpha_i(1-m)\fint_{\mathbb{T}}\frac{(b_j \overline w-b_i\overline\tau)\tau^{m}}{b_j w-b_i\tau}{}d\tau\\
&-&b_i\fint_{\mathbb{T}}\frac{(b_j \overline w-b_i\overline\tau)(\alpha_j \overline w^{m-1}-\alpha_i\overline\tau^{m-1})}{(b_j w-b_i\tau)^2}{}d\tau.
\end{eqnarray*}
For our purpose we do not need to compute exactly these integrals and scaling arguments appear in fact  to be sufficient. Indeed,  making change of variables we find real constants $a_{1,j}, b_{1,j}\in\RR$ such that
$$
\partial_tI(\phi_j(t,s,w)_{|t=s=0}= a_{1,j} w^{m-1}+ b_{1,j} \overline{w}^{m+1}.
$$
Similarly, differentiating with respect to the variable $s$ and repeating the same argument as before we get real constants $a_{2,j}, b_{2,j}\in \RR$ such that
\begin{equation*}
\partial_sI(\phi_j(t,s,w)_{|t=s=0}= a_{2,j} w^{m-1}+ b_{2,j} \overline{w}^{m+1}.
\end{equation*}
By the same way we find  real constants $a_{3,j}, b_{3,j},\in \RR$ with
\begin{equation*}
\partial_t\partial_sI(\phi_j(t,s,w)_{|t=s=0}= a_{3,j} \overline{w}+ b_{3,j} \overline{w}^{2m+1}.
\end{equation*}
Note that complete details about these computations will be given later in the proof of \mbox{Proposition \ref{lem76}.}
Now it is not difficult to check that
\begin{eqnarray*}
\partial_t\partial_s G_j\big(\lambda_m, tZ_m+s\widehat{Z}_m\big)_{|t,s=0}&=&\hbox{Im}\Bigg\{(1-m)\overline{w}^{m-1}\beta_j\,\partial_tI(\phi_j(t,s,w)_{|t=s=0}\\
&+&(1-m)\overline{w}^{m-1}\alpha_j\,\partial_sI(\phi_j(t,s,w)_{|t=s=0}\\&
+&b_j w \partial_t\partial_sI(\phi_j(t,s,w)_{|t=s=0}\Bigg\}.
\end{eqnarray*}
Consequently we find a vector $C_3\in\RR^2$ such that
$$
\partial_t\partial_s G\big(\lambda_m, tZ_m+s\widehat{Z}_m\big)_{|t,s=0}= C_3\, e_{2m}.
$$
Hence  we deduce from \eqref{project1} that
$$
Q\partial_t\partial_s G\big(\lambda_m, tZ_m+s\widehat{Z}_m\big)_{|t,s=0}=0.
$$
The proof of Lemma \ref{lemexp} is now complete.
\end{proof}

\subsubsection{Computation of $\partial_{tt}F_2(\lambda_m,0)$}\label{SectionQ} We want to establish the result of Proposition \ref{explicitval}-(4) which is the last point to check. First
recall from Proposition \ref{platit2} that
$$
\partial_{tt}F_2(\lambda_m,0)=\frac13\frac{d^3}{dt^3}QG(\lambda_m, tv_m)_{|t=0}+Q\partial_{ff}G(\lambda_m,0)\big[v_m,\widetilde{v}_m\big]
$$
where the explicit expression of  $\widetilde{v}_m$ is given in Lemma \ref{lemexp}. For $m=2$ the vector  $\widetilde{v}_2$ vanishes and because the second term  of the right-hand side is bilinear then  necessary it will be zero which implies
$$
\partial_{tt}F_2(\lambda_2,0)=\frac13\frac{d^3}{dt^3}QG(\lambda_2, tv_2)_{|t=0}.
$$
Making appeal to  Proposition \ref{prop77} stated later in the Subsection \ref{Seccw1}   and using  the \mbox{convention  \eqref{formm}} we readily  get 
\begin{eqnarray*}
\partial_{tt}F_2(\lambda_2,0)
&=&-\frac{(1-b^2)^2}{b}\sqrt{2}\,\mathbb{W}_2.
\end{eqnarray*}
However, 
for  $m\ge 3$ we combine Proposition \ref{prop77} with Proposition \ref{lem76} and \eqref{formm} leading to 
\begin{eqnarray*}
\partial_{tt}F_2(\lambda_m,0)&=&\frac{1}{\sqrt{2}}\Big[ m(m-1)\frac{(1-b^2)^2}{b}+{\widehat{\beta}_m \widehat{\gamma}_m}\Big]\mathbb{W}_m,
\end{eqnarray*}
with
$$
\widehat{\beta}_m =\frac{2m(b^2+b^{m})^2}{ (b^m+1)^2(-b^{2m}+2b^m+1)}
$$
and 
$$\widehat{\gamma}_m=(m-2)b+mb^{2m-1}+(4m-6)b^{m+1}
+4(m-1)b^{2m+1}+2mb^{3m-1}.
$$
The only point remaining concerns   the formulae for $\frac{d^3}{dt^3}QG(\lambda_2, tv_2)_{|t=0}$ and $Q\partial_{ff}G(\lambda_m,0)[v_m,\widetilde{v}_m]$  which will be done separately in the next two subsections.
\subsubsection{Computation of $\frac{d^3}{dt^3}QG(\lambda_m, tv_m)_{|t=0}$}\label{Seccw1}
The main result reads as follows.
\begin{proposition}\label{prop77}
Let $m\geq2$, then we have
$$
\frac13\frac{d^3}{dt^3}QG(\lambda_m, tv_m)_{|t=0}=-\varepsilon m(m-1)\frac{(1-b^2)^2}{b\sqrt{2}}\mathbb{W}_m,
$$
\end{proposition}
where the definition of $\varepsilon$ was given in \eqref{formm}.
%Let's remark that the above quantity is negative for $m=2$ and positive for $m\ge 3.$
\begin{proof}
From Leibniz's rule we can easily check that
\begin{equation}\label{FF11}
\frac{d^3}{dt^3}G_j(\lambda_m, tv_m)_{|t=0}= \hbox{Im}\bigg\{b_j w\frac{d^3}{dt^3}I(\phi_j(t,\cdot))_{|t=0}+3(1-m)\alpha_j \overline{w}^{m-1}\frac{d^2}{dt^2}(I(\phi_j(t,\cdot))_{|t=0} )\bigg\},
\end{equation}
where the function $G_j$ was defined in \eqref{G59}.
According to  \eqref{sd12} there exist real coefficients $\eta_{i,j}$ such that 
\begin{eqnarray*}
\frac{d^2}{dt^2}(I(\phi_1(t,\cdot))_{|t=0} &=&2(m-1)\big( b^2-1\big)\overline{w}+2(\eta_{1,1}-\eta_{2,1})\overline{w}^{2m+1},\\
\nonumber\frac{d^2}{dt^2}(I(\phi_2(t,\cdot))_{|t=0} &=&2(m-1)\big(\varepsilon b^{m-1}-b^{-1}\big)\overline{w}+2(\eta_{1,2}-\eta_{2,2})\overline{w}^{2m+1}.
\end{eqnarray*}
Thus what is left is  to compute  $\frac{d^3}{dt^3}I(\phi_j)_{t=0}$. With the notation introduced in \eqref{TQ1} we can verify that
$$
\frac{d^3}{dt^3}I_i(\phi_j)_{t=0}=-6\fint_{\mathbb{T}}\frac{A\overline{B}-\overline{A}B}{A^4}B(AC-b_iB)d\tau.
$$
Using a scaling argument as before  we find  real constants $\widehat\eta_{ij}$ such that 
\begin{eqnarray*}
\frac16\frac{d^3}{dt^3}I_i(\phi_j)_{t=0}&=&-\fint_{\mathbb{T}}\frac{B\overline{B} C}{A^2}d\tau+b_i\fint_{\mathbb{T}}\frac{B^2\overline{B}}{A^3}d\tau+\widehat\eta_{ij}\,\overline{w}^{3m+1}\\
&=&(m-1)\overline{w}^{m+1} J_{i,j}+b_i\overline{w}^{m+1} K_{i,j}+\widehat\eta_{ij} \,\overline{w}^{3m+1}
\end{eqnarray*}
with
$$
J_{i,j}=\alpha_i\fint_{\mathbb{T}}\frac{(\alpha_j-\alpha_i\overline{\tau}^{m-1})(\alpha_j-\alpha_i\tau^{m-1}) \overline{\tau}^m}{(b_j-b_i\tau)^2}
d\tau
$$
and
$$
K_{i,j}=\fint_{\mathbb{T}}\frac{(\alpha_j-\alpha_i\overline{\tau}^{m-1})^2(\alpha_j-\alpha_i\tau^{m-1})}{(b_j-b_i\tau)^3}d\tau.
$$
It is now a relatively simple matter to compute  $J_{i,j}$.  For the term $J_{1,2}$,  it is computed   by combining the notation   \eqref{Notqy1} with    residue theorem \mbox{at $\infty$} 
\begin{eqnarray}\label{Tu1}
\nonumber J_{1,2}&=&\varepsilon b\fint_{\mathbb{T}}\frac{(1-\varepsilon b\overline{\tau}^{m-1})(1-\varepsilon b\tau^{m-1}) \overline{\tau}^m}{(b-\tau)^2}
d\tau\\
&=&0.
\end{eqnarray}
With respect to   $J_{2,1}$ we write  by the definition 
\begin{eqnarray*}
J_{2,1}&=&\fint_{\mathbb{T}}\frac{1+b^2-\varepsilon b(\tau^{m-1}+\overline{\tau}^{m-1})}{(1-b\tau)^2}\overline{\tau}^md\tau\\
&=&\fint_{\mathbb{T}}\frac{(1+b^2)\overline{\tau}^m-\varepsilon b\overline{\tau}-\varepsilon b\overline{\tau}^{2m-1}}{(1-b\tau)^2}d\tau.
\end{eqnarray*}
Therefore we get with the help of  \eqref{form706}
\begin{equation}\label{Tu2}
J_{2,1}=(1+b^2) mb^{m-1}-\varepsilon b-\varepsilon (2m-1)b^{2m-1}.
\end{equation}
As to the diagonal terms $J_{i,i}$, one has
\begin{eqnarray}
\nonumber J_{i,i}&=&\frac{\alpha_i^2}{b_i^2}\fint_{\mathbb{T}}\frac{(1-\overline{\tau}^{m-1})(1-\tau^{m-1}) \overline{\tau}^m}{(1-\tau)^2}d\tau.
\end{eqnarray}
Because the integrand has a holomorphic continuation to  $\CC^\star$ , then we deduce by  residue theorem at $\infty$
\begin{equation}\label{Tu3}
J_{i,i}=0.
\end{equation}
Now we shall move to the computation of $K_{i,j}$. Using the definition and  residue theorem at $\infty$ we obtain by virtue of \eqref{Idgh1}
\begin{eqnarray*}%\label{Tu4}
\nonumber K_{1,2}&=&\fint_{\mathbb{T}}\frac{(1-\varepsilon b\overline{\tau}^{m-1})^2(1-\varepsilon b\tau^{m-1}) }{(b-\tau)^3}d\tau\\
\nonumber&=&\varepsilon b\fint_{\mathbb{T}}\frac{{\tau}^{m-1} }{(\tau-b)^3}d\tau\\
&=&\frac12\varepsilon (m-1)(m-2)b^{m-2}.
\end{eqnarray*}
In regard to  $K_{2,1}$  once again  we  invoke  residue theorem  which allows to get

\begin{eqnarray}%\label{Tu4}
\nonumber K_{2,1}&=&\fint_{\mathbb{T}}\frac{(\varepsilon b-\overline{\tau}^{m-1})^2(\varepsilon b-\tau^{m-1}) }{(1-b\tau)^3}d\tau\\
\nonumber&=&-(1+2b^2)\fint_{\mathbb{T}}\frac{\overline{\tau}^{m-1}}{(1-b\tau)^3}d\tau+\varepsilon b\fint_{\mathbb{T}}\frac{\overline{\tau}^{2m-2}}{(1-b\tau)^3}d\tau.
\end{eqnarray}
It follows from \eqref{Idgh1} that
\begin{equation}\label{Tu6}
K_{2,1}=(m-1)\Big[(2m-1) \varepsilon\,b^{2m-2}-mb^m-\frac12 m b^{m-2}\Big].
\end{equation}
For the term $K_{i,i}$ it suffices to use  residue theorem at $z=1$ as follows,
\begin{eqnarray*}
\nonumber K_{i,i}&=&\frac{\alpha_i^3}{b_i^3}\fint_{\mathbb{T}}\frac{(1-\overline{\tau}^{m-1})^2(1-\tau^{m-1}) }{(1-\tau)^3}d\tau\\
\nonumber&=&\frac{\alpha_i^3}{b_i^3}\fint_{C_r}\frac{3-\tau^{m-1}}{(1-\tau)^3}d\tau\\
&=&\frac12\frac{\alpha_i^3}{b_i^3}(m-1)(m-2),
\end{eqnarray*}
with $C_r$ the circle of center zero and radius $r>1.$  Consequently one deduces  by \eqref{Notqy1} 
\begin{equation*}%\label{Tu7}
K_{1,1}=\frac12\varepsilon b^3(m-1)(m-2);\quad K_{2,2}=\frac12b^{-3}(m-1)(m-2).
\end{equation*}
Putting together the preceding identities yields successively,
\begin{eqnarray*}%\label{Tu8}
\frac16\frac{d^3}{dt^3}I_1(\phi_1)_{t=0}=\frac12\varepsilon b^3(m-1)(m-2)\overline{w}^{m+1}+ \widehat\eta_{11}\overline{w}^{3m+1},
\end{eqnarray*}
\begin{eqnarray*}%\label{Tu9}
\frac16\frac{d^3}{dt^3}I_2(\phi_1)_{t=0}=(m-1)\Big(\frac12 m b^{m-1}-\varepsilon b\Big)\overline{w}^{m+1}+ \widehat\eta_{21}\overline{w}^{3m+1},
\end{eqnarray*}
\begin{eqnarray*}%\label{Tu10}
\frac16\frac{d^3}{dt^3}I_1(\phi_2)_{t=0}=\frac12\varepsilon (m-1)(m-2) b^{m-2}\overline{w}^{m+1}+ \widehat\eta_{12}\overline{w}^{3m+1},
\end{eqnarray*}
and
\begin{eqnarray*}%\label{Tu11}
\frac16\frac{d^3}{dt^3}I_2(\phi_2)_{t=0}=\frac12(m-1)(m-2) b^{-2}\overline{w}^{m+1}+  \widehat\eta_{22}\overline{w}^{3m+1}.
\end{eqnarray*}
Coming back to the definition of $I(\phi_j)$ and using the foregoing identities we may find
\begin{eqnarray*}%\label{Tu12}
\frac16\frac{d^3}{dt^3}I(\phi_1)_{t=0}=\frac12(m-1)\Big[\varepsilon b^3(m-2)-mb^{m-1}+2\varepsilon b\Big]\overline{w}^{m+1}+ 
(\widehat\eta_{11}-\widehat\eta_{21})\overline{w}^{3m+1},
\end{eqnarray*}
\begin{eqnarray*}%\label{Tu13}
\frac16\frac{d^3}{dt^3}I(\phi_2)_{t=0}=\frac12(m-1)(m-2)\Big(\varepsilon b^{m-2}-b^{-2}\Big)\overline{w}^{m+1}+ (\widehat\eta_{12}-\widehat\eta_{22})\overline{w}^{3m+1}.
\end{eqnarray*}
Inserting these identities into \eqref{FF11} yields 
\begin{eqnarray*}
\frac{d^3}{dt^3}G_1(\lambda_m, tv_m)_{|t=0}&=&  \hbox{Im}\bigg\{w\frac{d^3}{dt^3}I(\phi_1(t,\cdot))_{|t=0}+3(1-m)\varepsilon b \overline{w}^{m-1}\frac{d^2}{dt^2}(I(\phi_1(t,\cdot))_{|t=0} )_{|t=0}\bigg\}\\
&=& \hbox{Im}\bigg\{3m(m-1)\Big(2\varepsilon b-b^{m-1}-\varepsilon b^3\Big)\overline{w}^m+\gamma_1 \overline{w}^{3m}\bigg\}\\
&=&3m(m-1) \Big(2\varepsilon b-b^{m-1}-\varepsilon b^3\Big)\,e_m+\gamma_1 \,e_{3m}.
\end{eqnarray*}
Likewise, we get for some real constant $\gamma_2$
\begin{eqnarray*}
\frac{d^3}{dt^3}G_2(\lambda_m, tv_m)_{|t=0}&=&  \hbox{Im}\bigg\{b w\frac{d^3}{dt^3}I(\phi_2(t,\cdot))_{|t=0}+3(1-m) \overline{w}^{m-1}\frac{d^2}{dt^2}(I(\phi_2(t,\cdot))_{|t=0} )_{|t=0}\bigg\}\\
&=&-3m(m-1)\Big(\varepsilon b^{m-1}-b^{-1}\Big)\,e_m+\gamma_2 \,e_{3m},
\end{eqnarray*}
with $\gamma_1,\gamma_2\in \RR.$ Writing them in the vectorial form  
$$
\frac{d^3}{dt^3}G(\lambda_m, tv_m)_{|t=0}=3m(m-1) \begin{pmatrix}
2\varepsilon b-b^{m-1}-\varepsilon b^3\\
-\varepsilon b^{m-1}+b^{-1}
\end{pmatrix} e_m+\begin{pmatrix}
\gamma_1\\
\gamma_2
\end{pmatrix} e_{3m}.
$$
By means of  \eqref{project1} we find
\begin{eqnarray*}
\frac13\frac{d^3}{dt^3}Q G(\lambda_m, tv_m)_{|t=0}&=&m(m-1)\Big\langle \begin{pmatrix}
2\varepsilon b-b^{m-1}-\varepsilon b^3\\
-\varepsilon b^{m-1}+b^{-1}
\end{pmatrix}, \widehat{\mathbb{W}}\Big\rangle \mathbb{W}_m\\
&=&\frac{1}{\sqrt{2}} m(m-1)\varepsilon\Big(2b-b^3-b^{-1}\Big)\mathbb{W}_m\\
&=&-\varepsilon\,m(m-1)\frac{(1-b^2)^2}{b\sqrt2}\mathbb{W}_m
\end{eqnarray*}
and  the proposition follows.
\end{proof}
\subsubsection{Computation of $Q\partial_{ff}G(\lambda_m,0)[v_m,\widetilde{v}_m]$}
Now we intend to establish the following identity which has been used in the Subsection \ref{SectionQ}. %For the explicit value  of $\widetilde{v}_m$ we refer to \mbox{Lemma \ref{lemexp}.}
\begin{proposition}\label{lem76}
The following assertions hold true.
\begin{enumerate}
\item Case $m=2;$
$$
Q\partial_{ff}G(\lambda_m,0)[v_2,\widetilde{v}_2]=0.
$$
\item Case $m\geq3;$
\begin{equation*}
Q\partial_{ff}G(\lambda_m,0)[v_m,\widetilde{v}_m]=\frac{\widehat{\beta}_m \widehat{\gamma}_m}{\sqrt{2}}\mathbb{W}_m.
 \end{equation*}
 with
$$
\widehat{\beta}_m =\frac{2m(b^2+b^{m})^2}{ (b^m+1)^2(-b^{2m}+2b^m+1)}
$$
and 
$$\widehat{\gamma}_m=(m-2)b+mb^{2m-1}+(4m-6)b^{m+1}
+4(m-1)b^{2m+1}+2mb^{3m-1}.
$$
\end{enumerate}
\end{proposition}
\begin{proof}
$(\bf{1})$ This follows at once  from the  linearity of the derivation combined with \mbox{Lemma \ref{lemexp},}
\\
$(\bf{2})$
Recall from the definition  that 
$$
Q\partial_{ff}G(\lambda_m,0)[v_m,\widetilde{v}_m]=Q\partial_t\partial_s G(\lambda_m, tv_m+s\widetilde{v}_m)_{|t=s=0}.
$$
From the explicit value  of $\widetilde{v}_m$  given in  Lemma \ref{lemexp}  we can make the split $\widetilde{v}_m=\widehat{\beta}_m\widehat{v}_m$ with
$$
\widehat{\beta}_m =\frac{2m(b^2+b^{m})^2}{ (b^m+1)^2(-b^{2m}+2b^m+1)},\quad \widehat{v}_m\triangleq \left( \begin{array}{c}
\beta_1\\
\beta_2
\end{array} \right)\overline{w}^{2m-1} = \left( \begin{array}{c}
-1-2b^m \\
b^{2m-1}
\end{array} \right)  \overline{w}^{2m-1}.
$$ 
After factoring by $\widehat{\beta}_m$,  we get %Then by linearity we can factor out  as follows
$$
Q\partial_{ff}G(\lambda_m,0)[v_m,\widetilde{v}_m]=\widehat{\beta}_mQ\partial_t\partial_s G(\lambda_m, tv_m+s\widehat{v}_m)_{|t=s=0}.
$$
Let us recall  from \eqref{formm} and \eqref{kernexp}   that  for $m\ge 3$
$$
{v}_m=\left( \begin{array}{c}
\alpha_1 \\
\alpha_2
\end{array} \right)\overline{w}^{m-1}= \left( \begin{array}{c}
-b\\
1
\end{array} \right)\overline{w}^{m-1}.
$$
Next we shall define  the functions
$$
\varphi_j(t,s,w)\triangleq b_j w+t\alpha_j \overline{w}^{m-1}+s\beta_j \overline{w}^{2m-1}\quad \hbox{and}\quad G_j(t,s,w)\triangleq G_j(\lambda_m, tv_m+s\widehat{v}_m).
$$
Then, it is easily seen that $$
G_j(t,s,w)=\hbox{Im}\bigg\{\Big[(1-\lambda_m)\overline\varphi_j(t,s,w)+I(\varphi_j(t,s,w))\Big] w\Big(b_j+t(1-m)\alpha_j \overline{w}^m+s(1-2m)\beta_j \overline{w}^{2m}\Big)\bigg\}.
$$
Differentiating successively with respect to $t$ and $s$ according to Leibniz's rule we find
\begin{eqnarray}\label{G978}
\nonumber\partial_t\partial_sG_j(0,0,w)&=&\hbox{Im}\bigg\{(1-\lambda_m)\alpha_j\beta_j\big[(1-2m)\overline{w}^m+(1-m) w^m\big]+w b_j\partial_t\partial_s I(\varphi_j(0,0,w))\\
&+& (1-2m)\beta_j \overline{w}^{2m-1}\partial_t I(\varphi_j(0,0,w))+(1-m)\alpha_j \overline{w}^{m-1}\partial_sI(\varphi_j(0,0,w))\bigg\}.
\end{eqnarray}
We merely note that $I(\varphi_j)=I_1(\varphi_j)-I_2(\varphi_j)$ with 
\begin{eqnarray*}
I_i(\varphi_j(t,s,w))&=&\fint_{\mathbb{T}}\frac{\overline{\varphi_j(t,s,w)}-{\overline{\varphi_j(t,s,\tau)}}}{\varphi_j(t,s,w)-\varphi_j(t,s,\tau)}\varphi_i^\prime(t,s,\tau) d\tau\\
&=&\fint_{\mathbb{T}}\frac{\overline{A}+t\overline{B}+s\overline{C}}{A+tB+sC}\big[b_i+tD+sE\big]d\tau
\end{eqnarray*}
and where we adopt the notation :
\begin{eqnarray*}
A&=&b_j  w-b_i\tau,\quad B=\alpha_j \overline{w}^{m-1}-\alpha_i\overline{\tau}^{m-1}, \quad C=\beta_j \overline{w}^{2m-1}-\beta_i\overline{\tau}^{2m-1},\\
 D&=&(1-m)\alpha_i \overline{\tau}^m,\quad E=(1-2m)\beta_i \overline{\tau}^{2m}.
\end{eqnarray*}
  We shall start with calculating the term $\partial_t I_i(\varphi_j(0,0,w))$. Then plain computations lead to 
$$
\partial_t I_i(\varphi_j(0,0,w))=b_i\fint_{\mathbb{T}}\frac{A \overline{B}-\overline{A} B}{A^2} 
d\tau+\fint_{\mathbb{T}}\frac{\overline{A} D}{A} d\tau.
$$
As it was emphasized  before, because we use  a projection  $Q$ which captures  only the terms on $e_m$ then we do not need to know exactly all the terms but only some of them. This elementary fact will be often used in this part. Now by a scaling argument it comes that
$$
\partial_t I_i(\varphi_j(0,0,w))=b_i\fint_{\mathbb{T}}\frac{ \overline{B}}{A}d\tau+\theta_{i,j}  \overline{w}^{m+1},
$$
with $\theta_{i,j}\in \RR.$ Therefore coming back to the notation we may write
$$
\partial_t I_i(\varphi_j(0,0,w))=b_i w^{m-1}\fint_{\mathbb{T}}\frac{\alpha_j-\alpha_i\tau^{m-1}}{b_j-b_i\tau}d\tau +\theta_{i,j}
\overline{w}^{m+1}.
$$
As an immediate consequence we obtain successively by means of residue theorem,
$$
\partial_t I_i(\varphi_i(0,0,w))= \theta_{i,i} \overline{w}^{m+1},
$$
\begin{eqnarray*}
\partial_t I_1(\varphi_2(0,0,w))&=& w^{m-1}\fint_{\mathbb{T}}\frac{1+b\tau^{m-1}}{b-\tau}d\tau + \theta_{1,2} \overline{w}^{m+1}\\
&=&-w^{m-1}\big[1+b^m\big]+\theta_{1,2} \overline{w}^{m+1}
\end{eqnarray*} 
and 
\begin{eqnarray*}
\partial_t I_2(\varphi_1(0,0,w))&=& bw^{m-1}\fint_{\mathbb{T}}\frac{-b-\tau^{m-1}}{1-b\tau}d\tau+\theta_{2,1} \overline{w}^{m+1}\\
&=&\theta_{2,1} \overline{w}^{m+1}.
\end{eqnarray*}
As concerns the term  $\partial_sI_i(\varphi_j(0,0,w)),$ it is easy to check  by elementary  calculation that
$$
\partial_s I_i(\varphi_j(0,0,w))=b_i\fint_{\mathbb{T}}\frac{A \overline{C}-\overline{A} C}{A^2}
d\tau+\fint_{\mathbb{T}}\frac{\overline{A} E}{A} d\tau.
$$
As before, using  a  scaling argument  one finds real coefficients $\widehat\theta_{i,j}\in \RR$ such that
\begin{eqnarray*}
\partial_s I_i(\varphi_j(0,0,w))&=&b_i\fint_{\mathbb{T}}\frac{ \overline{C}}{A}d\tau+\widehat\theta_{i,j} \overline{w}^{2m+1}\\
&=&b_iw^{2m-1}\fint_{\mathbb{T}}\frac{\beta_j-\beta_i\tau^{2m-1}}{b_j-b_i\tau}d\tau+\widehat\theta_{i,j}  \overline{w}^{2m+1}.
\end{eqnarray*}
This implies in particular the vanishing of the diagonal terms in the integral, namely,
$$
\partial_s I_i(\varphi_i(0,0,w))=\widehat\theta_{i,i}  \overline{w}^{2m+1}.
$$
Applying once again  residue theorem we get
\begin{eqnarray*}
\partial_s I_2(\varphi_1(0,0,w))&=&bw^{2m-1}\fint_{\mathbb{T}}\frac{\beta_1-\beta_2\tau^{2m-1}}{1-b\tau}d\tau+\widehat\theta_{2,1}  \overline{w}^{2m+1}\\
&=&\widehat\theta_{2,1}  \overline{w}^{2m+1}.
\end{eqnarray*}
Repeating the same procedure we may write
\begin{eqnarray*}
\partial_s I_1(\varphi_2(0,0,w))&=&w^{2m-1}\fint_{\mathbb{T}}\frac{\beta_1\tau^{2m-1}-\beta_2}{\tau-b}d\tau+\widehat\theta_{1,2}\overline{w}^{2m+1}\\
&=&[\beta_1 b^{2m-1}-\beta_2]w^{2m-1}+\widehat\theta_{1,2} \overline{w}^{2m+1}\\
&=&-2b^{2m-1}\big[1+b^m\big]w^{2m-1}+\widehat\theta_{1,2}  \overline{w}^{2m+1}.
\end{eqnarray*}
Now, combining the preceding identities we find
\begin{eqnarray*}
\partial_t I(\varphi_1(0,0,w))&=&(\theta_{1,1}-\theta_{2,1}) \overline{w}^{m+1},\\
\partial_s I(\varphi_1(0,0,w))&=&\big(\widehat\theta_{1,1}-\widehat\theta_{2,1}\big) \overline{w}^{2m+1}
\end{eqnarray*}
and
\begin{eqnarray*}
 \partial_t I(\varphi_2(0,0,w))&=&-(1+b^m) w^{m-1}+(\theta_{1,2}-\theta_{2,2})\overline{w}^{m+1},\\
\partial_s I(\varphi_2(0,0,w))&=&-2b^{2m-1}(1+b^m) w^{2m-1}+\big(\widehat\theta_{1,2}-\widehat\theta_{2,2} \big)\overline{w}^{2m+1}.
\end{eqnarray*}
The next step is  to compute the mixed derivative $\partial_s\partial_t I_i(\phi_j(0,0))$ which is slightly more technical. Straightforward computations yield
\begin{eqnarray*}
\partial_t\partial_s I_i(\varphi_j(0,0))&=&b_i\fint_{\mathbb{T}}\frac{B \overline{C}-\overline{B} C}{A^2} d\tau-2b_i\fint_{\mathbb{T}}\frac{A \overline{C}-\overline{A} C}{A^3}B d\tau\\
&+&\fint_{\mathbb{T}}\frac{A \overline{B}-\overline{A} B}{A^2} Ed\tau+\fint_{\mathbb{T}}\frac{A \overline{C}-\overline{A} C}{A^2} Dd\tau.
\end{eqnarray*}
Performing  a scaling argument we can show the existence of real constants $\zeta_{i,j}$ such that 
\begin{eqnarray}\label{sec47}
\nonumber\partial_t\partial_s I_i(\varphi_j(0,0))&=&\zeta_{i,j}\overline{w}^{3m+1}-b_i\fint_{\mathbb{T}}\frac{B \overline{C}}{A^2} d\tau-b_i\fint_{\mathbb{T}}\frac{C\overline{B}}{A^2} d\tau\\
\nonumber &+&\fint_{\mathbb{T}}\frac{ \overline{B}E}{A} d\tau+\fint_{\mathbb{T}}\frac{D \overline{C}}{A} d\tau\\
&\triangleq&\zeta_{i,j}\overline{w}^{3m+1}-b_i w^{m-1}I_1^{i,j}-b_i \overline{w}^{m+1}I_2^{i,j}+\overline{w}^{m+1}I_3^{i,j}+w^{m-1}I_4^{i,j}.
\end{eqnarray}
To evaluate   $I_1^{i,j}$ we write by the definition
$$
I_1^{i,j}=\fint_{\mathbb{T}}\frac{\big(\alpha_j-\alpha_i\overline{\tau}^{m-1}\big)\big(\beta_j-\beta_i{\tau}^{2m-1}\big)}{(b_j-b_i\tau)^2}d\tau.
$$
For  $i=j$, we use successively  residue theorem at $\infty$ and $z=1$ leading for any   $r>1$ to 
\begin{eqnarray*}
I_1^{i,i}&=&\frac{\alpha_i\beta_i}{b_i^2}\Big[-\fint_{C_r}\frac{\tau^{2m-1}}{(1-\tau)^2}d\tau+\fint_{C_r}\frac{\tau^{m}}{(1-\tau)^2}d\tau\Big]\\
&=&\frac{\alpha_i\beta_i}{b_i^2}\big(-2m+1+m\big)\\
&=&\frac{\alpha_i\beta_i}{b_i^2}(1-m),
\end{eqnarray*}
with  $C_r=r\mathbb{T}$.
Concerning the  case $i=1, j=2$ we combine residue theorem at $\infty$ with \eqref{form706}
\begin{eqnarray*}
I_1^{1,2}&=&\fint_{\mathbb{T}}\frac{\big(\alpha_2-\alpha_1\overline{\tau}^{m-1}\big)\big(\beta_2-\beta_1{\tau}^{2m-1}\big)}{(b-\tau)^2}d\tau\\
&=&-\alpha_2\beta_1\fint_{\mathbb{T}}\frac{\tau^{2m-1}}{(\tau-b)^2}d\tau+\alpha_1\beta_1\fint_{\mathbb{T}}\frac{\tau^{m}}{(\tau-b)^2}d\tau\\
&=&-\alpha_2\beta_1(2m-1) b^{2m-2}+\alpha_1\beta_1 mb^{m-1}\\
&=&-\beta_1\big[(2m-1) b^{2m-2}+mb^{m}\big].
\end{eqnarray*}
As to the   case $i=2, j=1$ we use  \eqref{form706} 
\begin{eqnarray*}
I_1^{2,1}&=&\fint_{\mathbb{T}}\frac{\big(\alpha_1-\alpha_2\overline{\tau}^{m-1}\big)\big(\beta_1-\beta_2{\tau}^{2m-1}\big)}{(1-b\tau)^2}d\tau\\
&=&-\alpha_2\beta_1\fint_{\mathbb{T}}\frac{\overline{\tau}^{m-1}}{(1-b\tau)^2}d\tau\\
&=&-\beta_1(m-1) b^{m-2}.
\end{eqnarray*}
Now we move to the term  $I_2^{i,j}$ and  use first the  change of variables $\tau\mapsto \overline\tau$
\begin{eqnarray*}
I_2^{i,j}&=&\fint_{\mathbb{T}}\frac{\big(\alpha_j-\alpha_i{\tau}^{m-1}\big)\big(\beta_j-\beta_i\overline{\tau}^{2m-1}\big)}{(b_j-b_i\tau)^2}d\tau\\
&=&
\fint_{\mathbb{T}}\frac{\big(\alpha_j-\alpha_i\overline{\tau}^{m-1}\big)\big(\beta_j-\beta_i{\tau}^{2m-1}\big)}{(b_i-b_j\tau)^2}d\tau, 
\end{eqnarray*}
Thus similarly to $I_{1}^{2,1}$ and  $I_{1}^{1,2}$ we get
\begin{eqnarray*}
I_2^{1,2}&=&\fint_{\mathbb{T}}\frac{\big(\alpha_2-\alpha_1\overline{\tau}^{m-1}\big)\big(\beta_2-\beta_1{\tau}^{2m-1}\big)}{(1-b\tau)^2}d\tau\\
&=&-\alpha_1\beta_2\fint_{\mathbb{T}}\frac{\overline{\tau}^{m-1}}{(1-b\tau)^2}d\tau\\
&=&\beta_2(m-1) b^{m-1}.
\end{eqnarray*}
By the same way we deduce that
\begin{eqnarray*}
I_2^{2,1}&=&\fint_{\mathbb{T}}\frac{\big(\alpha_1-\alpha_2\overline{\tau}^{m-1}\big)\big(\beta_1-\beta_2{\tau}^{2m-1}\big)}{(b-\tau)^2}d\tau\\
%&=&-\alpha_1\beta_2\fint_{\mathbb{T}}\frac{\tau^{2m-1}}{(\tau-b)^2}d\tau+\alpha_2\beta_2\fint_{\mathbb{T}}\frac{\tau^{m}}{(\tau-b)^2}d\tau\\
&=&-\alpha_1\beta_2(2m-1) b^{2m-2}+\alpha_2\beta_2 mb^{m-1}\\
&=&\beta_2\big[(2m-1) b^{2m-1}+mb^{m-1}\big].
\end{eqnarray*}
As regards  the diagonal case $i=j$, one can readily  write
\begin{eqnarray*}
I_2^{i,i}&=&I_1^{i,i}\\
&=&\frac{\alpha_i\beta_i}{b_i^2}(1-m).
\end{eqnarray*}
Let us now move to  the computation of the term $I_3^{i,j}$ which is given by
$$
I_3^{i,j}=(1-2m)\beta_i\fint_{\mathbb{T}}\frac{\alpha_j\overline{\tau}^{2m}-\alpha_i\overline{\tau}^{m+1}}{b_j-b_i\tau}d\tau
$$
First notice that from residue theorem
$$
I_3^{i,i}=0=I_3^{1,2}.
$$
To calculate the value of $I_3^{2,1}$ we shall use  \eqref{form76}
\begin{eqnarray*}
I_3^{2,1}&=&(1-2m)\beta_2\fint_{\mathbb{T}}\frac{\alpha_1\overline{\tau}^{2m}-\alpha_2\overline{\tau}^{m+1}}{1-b\tau}d\tau\\
&=&(1-2m)\beta_2\Big[\alpha_1 b^{2m-1}-\alpha_2 b^{m}\Big]\\
&=&(2m-1)\beta_2\big[b^{2m}+ b^{m}\big]\
\end{eqnarray*}
The last term that we shall deal with is  $I_4^{i,j}$ whose expression  is  given by
$$
I_4^{i,j}=(1-m)\alpha_i\fint_{\mathbb{T}}\frac{\beta_j\overline{\tau}^{m}-\beta_i{\tau}^{m-1}}{b_j-b_i\tau}d\tau.
$$
For  the diagonal terms $i=j$, we get for $r>1$ and $m\in \NN^\star$
\begin{eqnarray*}
I_4^{i,i}&=&(1-m)\frac{\alpha_i\beta_i}{b_i}\fint_{\mathbb{T}}\frac{\overline{\tau}^{m}-{\tau}^{m-1}}{1-\tau}d\tau\\
&=&(1-m)\frac{\alpha_i\beta_i}{b_i}\fint_{C_r}\frac{\overline{\tau}^{m}-{\tau}^{m-1}}{1-\tau}d\tau\\
&=&(1-m)\frac{\alpha_i\beta_i}{b_i}\fint_{C_r}\frac{{\tau}^{m-1}}{\tau-1}d\tau\\
&=&(1-m)\frac{\alpha_i\beta_i}{b_i}\cdot
\end{eqnarray*}
For the case   $i=1,j=2$, we write
\begin{eqnarray*}
I_4^{1,2}&=&(1-m)\alpha_1\fint_{\mathbb{T}}\frac{\beta_2\overline{\tau}^{m}-\beta_1{\tau}^{m-1}}{b-\tau}d\tau\\
&=&(1-m)\alpha_1\beta_1\fint_{\mathbb{T}}\frac{{\tau}^{m-1}}{\tau-b}d\tau\\
&=&(m-1)b^m\beta_1
\end{eqnarray*}
The last term  $I_4^{2,1}$ can be computed with the help of  \eqref{form76}
\begin{eqnarray*}
I_4^{2,1}&=&(1-m)\alpha_2\fint_{\mathbb{T}}\frac{\beta_1\overline{\tau}^{m}-\beta_2{\tau}^{m-1}}{1-b\tau}d\tau\\
&=&(1-m)\alpha_2\beta_1\fint_{\mathbb{T}}\frac{\overline{\tau}^{m}}{1-b\tau}d\tau\\
&=&(1-m)b^{m-1}\beta_1.
\end{eqnarray*}
Now let us gather the preceding identities in order to evaluate  $\partial_t\partial_s I_i(\varphi_j(0,0,w))$ which is defined in \eqref{sec47},
\begin{eqnarray*}
 \partial_t\partial_s I_i(\varphi_i(0,0,w))&=&\zeta_{i,i}\,\, \overline{w}^{3m+1}+w^{m-1}\big(-b_i I_1^{i,i}+I_4^{i,i}\big)+\overline{w}^{m+1}\big(-b_i I_2^{i,i}+I_3^{i,i}\big)\\
 &=&\zeta_{i,i}\,\, \overline{w}^{3m+1}+\frac{\alpha_i \beta_i}{b_i} (m-1)\overline{w}^{m+1}.
 \end{eqnarray*}
 Thus we find
 $$
  \partial_t\partial_s I_1(\varphi_1(0,0,w))=(m-1)b(1+2b^m) \overline{w}^{m+1}+ \zeta_{1,1}\,\, \overline{w}^{3m+1}
  $$
and
  $$
  \partial_t\partial_s I_2(\varphi_2(0,0,w))=b^{2m-2}(m-1) \overline{w}^{m+1}+\zeta_{2,2}\,\, \overline{w}^{3m+1}.
 $$
Furthermore, we obtain
\begin{eqnarray*}
 \partial_t\partial_s I_1(\varphi_2(0,0,w))&=&w^{m-1}\big(- I_1^{1,2}+I_4^{1,2}\big)+\overline{w}^{m+1}\big(- I_2^{1,2}+I_3^{1,2}\big)+\zeta_{1,2}\,\, \overline{w}^{3m+1}\\
 &=&\beta_1(2m-1)(b^m+b^{2m-2})w^{m-1}+\beta_2(1-m) b^{m-1}\overline{w}^{m+1}+\zeta_{1,2}\,\, \overline{w}^{3m+1}\\
 &=&(1-2m)(1+2b^m)(b^m+b^{2m-2})w^{m-1}+b^{3m-2}(1-m)\overline{w}^{m+1}+\zeta_{1,2}\,\, \overline{w}^{3m+1} \end{eqnarray*}
 and 
 \begin{eqnarray*}
 \partial_t\partial_s I_2(\varphi_1(0,0,w))&=&w^{m-1}\big(- bI_1^{2,1}+I_4^{2,1}\big)+\overline{w}^{m+1}\big(-b I_2^{2,1}+I_3^{2,1}\big)+\zeta_{2,1}\,\, \overline{w}^{3m+1}
\\
 &=&\beta_2(m-1) b^m\overline{w}^{m+1}+\zeta_{2,1}\,\, \overline{w}^{3m+1}
\\
  &=&(m-1) b^{3m-1}\overline{w}^{m+1}+\zeta_{2,1}\,\, \overline{w}^{3m+1}.
 \end{eqnarray*}
Consequently,
 \begin{eqnarray*}
 \partial_t\partial_s I(\varphi_1(0,0,w))&=&\partial_t\partial_s I_1(\varphi_1(0,0))-\partial_t\partial_s I_2(\varphi_1(0,0))\\
 &=&(m-1)\big(b+2b^{m+1}-b^{3m-1}\big)\overline{w}^{m+1}+(\zeta_{1,1}-\zeta_{2,1})\overline{w}^{3m+1}
 \end{eqnarray*}
 and
  \begin{eqnarray*}
 \partial_t\partial_s I(\varphi_2(0,0,w))&=&\partial_t\partial_s I_1(\varphi_2(0,0))-\partial_t\partial_s I_2(\varphi_2(0,0))\\
 &=&(1-2m)(1+2b^m)(b^m+b^{2m-2}) w^{m-1}+(1-m)\big(b^{3m-2}+b^{2m-2}\big)\overline{w}^{m+1}\\
 &+&(\zeta_{1,2}-\zeta_{2,2})\overline{w}^{3m+1}.
 \end{eqnarray*}
 Inserting the preceding identities into \eqref{G978} we find real constant $\widehat\zeta_1$ such that
  \begin{eqnarray*}
 \partial_t\partial_s G_1(0,0,w)&=&\widehat{\beta}_m\, \hbox{Im}\Big\{\overline{w}^m\, \Big((m-1)(b+2b^{m+1}-b^{3m-1})+(1-\lambda_m)(1-2m)(b+2b^{m+1})\Big)\\
 &+&w^m(1-m)(1-\lambda_m)(b+2b^{m+1})\Big\}+\widehat\zeta_{1}\, e_{3m}
 \end{eqnarray*}
 Since $\lambda_m=\frac{1+b^2}{2}$ then the preceding identity can be written in the form
  \begin{eqnarray*}
 \partial_t\partial_s G_1(0,0,w)&=\widehat{\beta}_m\Big[\big(m\frac{b^2-1}{2}+m-1
\big)\big(b+2b^{m+1}\big)-(m-1)b^{3m-1}\Big]e_m+\widehat\zeta_{1}\, e_{3m},
 \end{eqnarray*}
 with 
 $$
\widehat{\beta}_m =\frac{2m(b^2+b^{m})^2}{ (b^m+1)^2(-b^{2m}+2b^m+1)}\cdot
$$
 From the identity $m(1-b^2)=2+2b^m$ given in \eqref{relat1} we deduce that
\begin{eqnarray*}
 \partial_t\partial_s G_1(0,0,w)&=&\widehat{\beta}_m\Big[(b+2b^{m+1})(m-2-b^m)-(m-1)b^{3m-1}\Big]e_m+\widehat\zeta_{1}\, e_{3m}\\ &=&
\widehat{\beta}_m\Big[-(m-1)b^{3m-1}+(m-2)b-2b^{2m+1}+(2m-5)b^{m+1}\Big]+\widehat\zeta_{1}\, e_{3m}.
 \end{eqnarray*}
By the same way we can easily check that
 \begin{eqnarray*}
 \partial_t\partial_s G_2(0,0,w)&=&\widehat{\beta}_m\Big[(2m-1)b^{m+1}+mb^{2m-1}+(4m-2)b^{2m+1}+(3m-1) b^{3m-1}\Big] e_m\\
 &+&\widehat\zeta_{2}\, e_{3m},
  \end{eqnarray*}
with $\widehat\zeta_{2} \in \RR$. Using \eqref{project1} we obtain 
  \begin{equation*}
Q\partial_{ff}G(\lambda_m,0)[v_m,\widetilde{v}_m]=\frac{\widehat{\beta}_m}{\sqrt{2}}\Big\{(m-2)b+mb^{2m-1}+(4m-6)b^{m+1}
+4(m-1)b^{2m+1}+2mb^{3m-1}\Big\}\mathbb{W}_m.
 \end{equation*}
 This achieves the proof of the desired identity.
\end{proof}
%%
%%\section{Proof of the Main theorem}\label{Sec55}
%%
\section{Proof of the Main theorem}\label{Sec55}
This last section is devoted to the proof of our main result. \begin{proof} 
$(\bf{1})$
As we have seen in Section \ref{Sec2} the existence of m-fold  V-states  for $\lambda$ close enough to $\lambda_m$ reduces to solving the bifurcation equation \eqref{Bifur1} which is a   two-dimensional equation 
\begin{equation}\label{dfs1}
F_2(\lambda,t)=0,\quad (\lambda,t)\quad \hbox{close to}\quad (\lambda_m,0)
\end{equation}
Now we shall check the existence of nontrivial solutions to this equation for $m=2$ \mbox{and $b\notin\big\{b_{2n}, n\geq2\big\}$. }
Since $F_2$ is smooth then by Taylor expansion around the point $(\lambda_2,0)$ we get

\begin{eqnarray*}
F_2(\lambda,t)=\frac12(\lambda-\lambda_2)^2\partial_{\lambda\lambda} F_2(\lambda_2,0)+\frac12t^2\partial_{tt}F_2(\lambda_2,0)+\big((\lambda-\lambda_2)^2+t^2\big)\epsilon(\lambda, t)
\end{eqnarray*}
with
$$
\lim_{(\lambda,t)\to(\lambda_2,  0)}\epsilon(\lambda, t)=0.
$$
From  \mbox{Proposition \ref{explicitval}}  we deduce that
\begin{eqnarray*}
F_2(\lambda,t)=\Big[2\frac{\sqrt{2}}{b}(\lambda-\lambda_2)^2-\frac{\sqrt{2}(1-b^2)^2}{b}t^2\Big]\,\mathbb{W}_2+\big((\lambda-\lambda_2)^2+t^2\big)\epsilon(\lambda, t).
\end{eqnarray*}
Without loss of generality  we can assume that $\lambda_2=0$; it suffices to make dilation and translation, different in any variable.   It follows at once that  $F_2$ could be written in the form
\begin{equation*}
F_2(\lambda,t)=(\lambda^2-t^2)\,\mathbb{W}_2+\big(\lambda^2+t^2\big)\epsilon(\lambda, t).
\end{equation*}
Introducing the new variable $s$ such that  $\lambda= ts$ this equation  becomes
\begin{equation*}
\widehat F_2(t,s)\triangleq(s^2-1)\,\mathbb{W}_2+\big(s^2+1\big)\epsilon(ts, t)=0.
\end{equation*}
It is plain that
$$
\widehat F_2(0,1)=0.
$$
Moreover, by the definition and the fact $\epsilon(0,0)=0$ we easily find
\begin{eqnarray*}
\partial_s\widehat F_2(0,1)&=&\lim_{s\to 1}\frac{\widehat F_2(0,s)-\widehat F_2(0,1)}{s-1}\\
&=&2\,\mathbb{W}_2\neq0.
\end{eqnarray*}
Therefore we conclude by applying the implicit function theorem that near the point  $(0,1)$ the solutions of $\widehat F_2(t,s)=0$ can be parametrized by a simple curve of the form
$t\in ]-\alpha,\alpha[\mapsto (t, \mu(t))$ \mbox{with $\alpha>0$} is a small real number and $\mu(0)=1$. This guarantees the existence of a curve of nontrivial solution to \eqref{dfs1}. 
\\
We merely observe that by setting $\lambda= -ts$ we can repeat the preceding arguments and  show the existence of nontrivial curve of solutions to $\widehat F_2(t,s)=0$. Therefore the set  of solutions for  the equation \eqref{dfs1}   is described close to $(\lambda_2,0)$ by the union of two simple curves passing through this point. However from geometrical point of view these curves describe the same domains. Indeed,  let $D=D_1\backslash D_2$ be a two-fold doubly-connected V-state whose boundary $\Gamma_1\cup\Gamma_2$ is parametrized by the conformal mappings
$$
\phi_j(w)=w\Big(b_j+\sum_{n\geq 1}\frac{a_{j,n}}{w^{2n}}\Big),\quad  a_{j,n}\in \RR.
$$
Now denote by $\widehat D= iD$ the  rotation  of $D$ with the  angle $\frac\pi2$. Then the new domain   is also a V-state with two-fold structure and its conformal parametrization, which belongs to the affine space  $(b_1,b_2)\hbox{Id}+X_2$ ( recall that the space $X_m$ was introduced in \eqref{X_m}), is given by
$$
\widehat \phi_j(w)=\frac{1}{i}\phi_j(iw)=w\Big(b_j+\sum_{n\geq 1}(-1)^n\frac{a_{j,n}}{w^{2n}}\Big),\quad  a_{j,n}\in \RR.
$$
Thus from analytical point of view if $(\phi_1,\phi_2)$ is a nontrivial solution of the V-state problem then $(\widehat \phi_1,\widehat \phi_2)$ is a solution too and it  is different from the preceding one and rotating with the same angular velocity. This implies that in  the bifurcation diagram, if $(\phi_1,\phi_2)$  belongs to a given curve of bifurcation among those constructed  before (note that each curves  is transcritical) then $(\widehat\phi_1,\widehat\phi_2)$ belongs necessary to the second one. Otherwise, we should get four curves of bifurcation and not two. This gives the correspondence between the bifurcating curves which in fact describe the same geometry.\\ Finally, we point out that  $\lambda$ is parametrized as follows $\lambda(t)=t\mu(t)$ with $ t\in ]-\alpha,\alpha[$ and $\mu(0)=1.$ Therefore the range of $\lambda$ contains strictly zero which implies that the bifurcation is transcritical. 
\\
\\
$(\bf{2})$ Let $m\geq3$ and $b=b_m$ then using  Taylor expansion around the point $(\lambda_m,0)$ combined with \mbox{Proposition \ref{explicitval}} we get
\begin{eqnarray*}
F_2(\lambda,t)&=&\frac12(\lambda-\lambda_m)^2\,\partial_{\lambda\lambda} F_2(\lambda_m,0) +\frac12 t^2\,\partial_{tt}F_2(\lambda_m,0) +\big((\lambda-\lambda_m)^2+t^2\big)\epsilon(\lambda, t)\\
&=&\frac{1}{\sqrt{2}}\Big( m^2 b^{1-m}(\lambda-\lambda_m)^2+\frac{\widehat\alpha_m}{\sqrt{2}}\, t^2\Big)\mathbb{W}_m+\big((\lambda-\lambda_m)^2+t^2\big)\epsilon(\lambda, t)
\end{eqnarray*}
with
$$
\lim_{(\lambda,t)\to(\lambda_m,  0)}\epsilon(\lambda, t)=0.
$$
As $\widehat\alpha_m>0$ then the quadratic form associated to $F_2$ is definite positive meaning that $F_2$ is strictly convex around $(\lambda_m,0)$ and therefore in this neighborhood the equation \eqref{dfs1} has no other roots than the trivial one  $(\lambda_m,0).$ Consequently we deduce that for $\lambda$ close to $\lambda_m$ there is no  $m-$fold V-states which bifurcating from  the annulus $\mathbb{A}_b.$ This achieves  the proof of the Main theorem.
 
 \end{proof}
 
\begin{ackname}
We are indebted  to F. de la Hoz for the numerical experiments discussed in the Introduction, especially Figure \ref{flambda00}, Figure \ref{flambda1} and Figure \ref{fbvs}.  Taoufik Hmidi  was partially supported by the ANR project Dyficolti ANR-13-BS01-0003- 01. Joan Mateu  was partially supported by the grants of Generalitat de Catalunya 2014SGR7,
Ministerio de Economia y Competitividad MTM 2013-4469,
ECPP7- Marie Curie project MAnET. \end{ackname}

\end{document}